\documentclass{amsart}
\usepackage{amssymb}
\usepackage{graphicx}
\usepackage{epstopdf}
\usepackage{pdfsync}

\numberwithin{equation}{section}

\theoremstyle{plain} 
\newtheorem{thm}[equation]{Theorem}

\newtheorem{lem}[equation]{Lemma}
\newtheorem{prop}[equation]{Proposition}
\newtheorem{claim}[equation]{Claim}

\theoremstyle{definition}
\newtheorem{defn}[equation]{Definition}

\theoremstyle{remark}
\newtheorem{rem}[equation]{Remark}

\graphicspath{{PSF_figures/}}
\title[Tunnel-number-one knot exteriors in $S^3$]{Tunnel-number-one knot exteriors in $S^3$ disjoint from proper power curves}
\author{Sungmo Kang}

\email{skang4450@chonnam.ac.kr}

\begin{document}







\begin{abstract}
As one of the background papers of the classification project of hyperbolic primitive/Seifert knots in $S^3$ whose complete list is given in \cite{BK20},
this paper classifies all possible R-R diagrams of two disjoint simple closed curves $R$ and $\beta$ lying in the boundary of a genus two handlebody $H$ up to equivalence such that $\beta$ is a proper power curve and a 2-handle addition $H[R]$ along $R$ embeds in $S^3$ so that $H[R]$ is the exterior of a tunnel-number-one knot. As a consequence, if $R$ is a nonseparating simple closed curve on the boundary of a genus two handlebody such that $H[R]$ embeds in $S^3$, then there exists a proper power curve disjoint from $R$ if and only if $H[R]$ is the exterior of the unknot, a torus knot, or a tunnel-number-one cable of a torus knot.

The results of this paper will be mainly used in proving the hyperbolicity of P/SF knots and in classifying P/SF knots in once-punctured tori in $S^3$, which is one of the types of P/SF knots in \cite{BK20}. Together with these results, the preliminary of this paper which consists of three parts: the three diagrams which are Heegaard diagrams, R-R diagrams, and hybrid diagrams, `the Culling Lemma', and locating waves into an R-R diagrams, will also be used in the classification of hyperbolic primitive/Seifert knots in $S^3$.
\end{abstract}

\maketitle


\section{Introduction}\label{Introduction and main result}

In \cite{B90} or an available version \cite{B18}, Berge constructed twelve families of knots that admit lens space surgeries. These knots are referred to as the Berge knots and described in terms of double-primitive or primitive/primitive(or simply P/P) curves. A P/P curve $\alpha$ is a simple closed curve lying in a genus two Heegaard surface $\Sigma$ of $S^3$ bounding two handlebodies $H$ and $H'$ such that $\alpha$ is primitive in both $H$ and $H'$, i.e., 2-handle additions $H[\alpha]$ and $H'[\alpha]$ along $\alpha$  are solid tori. Then we say that a knot $k$ represented by $\alpha$ is a P/P knot and $(\alpha, \Sigma)$ is a P/P position of $k$. Note that a knot may have more than one P/P position.

For a surface-slope $\gamma$, which is defined to be an isotopy class of $\partial N(k)\cap \Sigma$, where $N(k)$ is a tubular neighborhood of $k$ in $S^3$, the knot $k$ represented by $\alpha$ admits lens space Dehn surgery. The Berge conjecture claims that the Berge knots cover all knots in $S^3$ admitting lens space Dehn surgeries. Much progress for this conjecture has been made, but it is still unsolved. One result toward this conjecture is that all P/P knots are the Berge knots, which is proved in \cite{B08} or independently in \cite{G13}. This implies that the Berge knots are the complete list of P/P knots.

In \cite{D03}, Dean introduced primitive/Seifert(or simply P/SF) knots, which is a natural generalization of P/P knots. A primitive/Seifert curve $\alpha$ is a simple closed curve lying in a genus two Heegaard surface $\Sigma$ of $S^3$ bounding two handlebodies $H$ and $H'$ such that $\alpha$ is primitive in one handlebody, say, $H'$, and $\alpha$ is Seifert in $H$, that is to say, $H'[\alpha]$ is a solid torus and $H[\alpha]$ is a Seifert-fibered space. Similarly as in P/P knots, we say that a knot $k$ represented by $\alpha$ is a P/SF knot and $(\alpha, \Sigma)$ is a P/SF position of $k$. Also for a surface-slope $\gamma$, since $k(\gamma)\cong H[\alpha]\cup_\partial H'[\alpha]$ by Lemma 2.3 of \cite{D03}, $\gamma$-Dehn surgery $k(\gamma)$ is either a Seifert-fibered space over $S^2$ with at most three exceptional fibers or a Seifert-fibered space over $\mathbb{R}P^2$ with at most two exceptional fibers.
Note that by \cite{EM92}, a connected sum of lens spaces can not arise as a Dehn surgery $k(\gamma)$ for hyperbolic P/SF knots.

P/SF knots are of interest, because knots with Dehn surgeries yielding Seifert-fibered spaces are not well understood. Some examples of P/SF knots are given in \cite{D03}, \cite{MM05}, and \cite{EM02}. However, as the classification of P/P knots is complete, there is a project to classify all hyperbolic P/SF knots. This project has been worked for years and has recently been achieved. See \cite{BK20} for the complete list of hyperbolic P/SF knots along with the surface-slope of the exceptional surgery on each knot that yields a Seifert-fibered space, and the indexes of each exceptional fiber in the resulting Seifert-fibered space.

The project for the classification of hyperbolic P/SF knot requires various backgrounds. This paper provides one of the background materials.
Now we explain the results of this paper. Let $R$ be a nonseparating simple closed curve on the boundary of a genus two handlebody $H$. Suppose that a 2-handle addition $H[R]$ along $R$ embeds in $S^3$, i.e., $H[R]$ is an exterior of a knot $k$ in $S^3$. Then it follows that $k$ is a tunnel-number-one knot in $S^3$ whose tunnel
is the cocore of the 2-handle. Suppose $\beta$ is another simple closed curve in $\partial H$ disjoint from $R$ which is a proper power curve. A proper power curve is defined to be a simple closed curve in $\partial H$ such that it is disjoint from an essential separating disk in $H$, does not bound a disk in $H$, and is not primitive in $H$.

The goal of this paper is to classify such curves $R$ and $\beta$ in terms of R-R diagrams. R-R diagrams are one way of describing simple closed curves on the boundary of a genus two handlebody. They are a type of planar diagram related to Heegaard diagrams. They are originally introduced by Osborne and Stevens in \cite{OS77} and developed by Berge. Definition and properties of genus $g$ R-R diagrams are given in Section~\ref{R-R diagrams}. Then the main results of this paper are as follows.

\begin{thm}\label{main theorem 1}
Suppose $R$ and $\beta$ are disjoint simple closed curves on the boundary of a genus two handlebody $H$ such that $H[R]$ embeds in $S^3$ and $\beta$ is a proper power curve.
Then $R$ and $\beta$ have an R-R diagram with the form shown either in Figure~\emph{\ref{DPCFig21b-1}a} with $s>1$ or Figure~\emph{\ref{DPCFig21b-1}b} with $m, n > 0, s>1$, $\gcd(m,n) = 1$ such that the set of parameters $(a,b,m,n,s)$ satisfies the condition \emph{(1), (2),} or \emph{(3)} below.
\begin{enumerate}
  \item $a + b = 1$ and, \emph{(}assuming without loss that $(a,b)$ = $(1,0)$\emph{)}, $sm - un = \delta$ for some $u$ with $0<u<s$, which implies that $H[R]$ is the exterior of an $(n, s)$ torus knot in $S^3$.
  \item $a,b > 0$, $m = 1$, and $s(a+b) - u[(a+b)n +b] = \delta$, which implies that $H[R]$ is the exterior of an $((a+b)n +bm, s)$ torus knot in $S^3$.
  \item $a,b > 0$, $n> m > 1$, $u = 1$, and $ms(a+b) - [(a+b)n +bm] = \delta$, which implies that $H[R]$ is the exterior of an $(ms(a+b) \pm1,s)$ cable about an $(a+b,m)$ torus knot in $S^3$ and a separating essential annulus bounded by two parallel copies of the curve $\beta$ is the cabling annulus in $H[R]$.
\end{enumerate}
\end{thm}

\begin{thm}
\label{main theorem2}
Suppose $R$ is a nonseparating simple closed curve on the boundary of a genus two handlebody such that $H[R]$ embeds in $S^3$. Then there exists a proper power curve disjoint from $R$ if and only if $H[R]$ is the exterior of the unknot, a torus knot, or a tunnel-number-one cable of a torus knot.
\end{thm}

\begin{figure}[tbp]
\centering
\includegraphics[width = 0.7\textwidth]{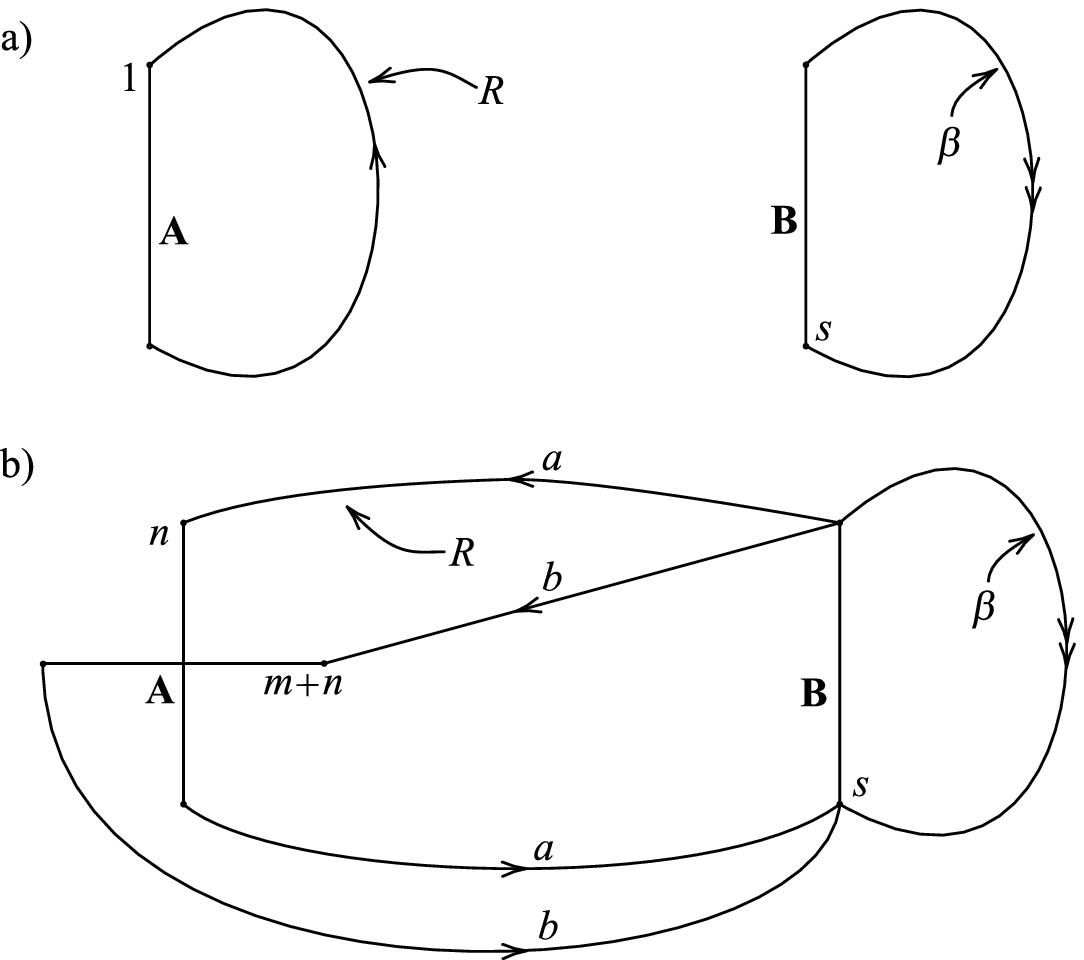}
\caption{All possible R-R diagrams of disjoint simple closed curves $R$ and $\beta$ up to equivalence on $H$ such that $H[R]$ embeds in $S^3$ and $\beta$ is a proper power curve.}
\label{DPCFig21b-1}
\end{figure}

The results of this paper are used in the classification project of hyperbolic P/SF knots in $S^3$. Especially, they are essential in proving the hyperbolicity of P/SF knots and in classifying P/SF knots in once-punctured tori in $S^3$, which is one of the types of P/SF knots in \cite{BK20}. Briefly speaking for the application, if $k$ has a P/SF position $(\alpha, \Sigma)$ such that $\alpha$ is primitive in $H'$ and is Seifert in $H$, then since $\alpha$ is primitive in $H'$, there exists a complete set of cutting disks $\{D_M, D_R\}$ of $H'$ such that the boundary $M$ of $D_M$ intersects $\alpha$ once and the boundary $R$ of $D_R$ is disjoint from $\alpha$. This implies that $M$ is a meridional curve of $k$ and $H[R]$ is homeomorphic
to the exterior of $k$ in $S^3$, and furthermore $k$ is a tunnel-number-one knot such that $R$ is the boundary of a cocore of the 1-handle regular neighborhood of a tunnel. Therefore we can conclude that if $k$ has a P/SF position $(\alpha, \Sigma)$, then there exists a simple closed curve $R$ in the boundary of $H$ such that $H[R]$ embeds in $S^3$ as the exterior of a tunnel-number-one knot. Theorem~\ref{main theorem 1} indicates that there is a close relationship between such a curve $R$ and the existence of a proper power curve.

Together with the results of this paper, the preliminary of this paper will also be used in the classification project of hyperbolic P/SF knots in $S^3$. It consists of three parts: three diagrams that are Heegaard diagrams, R-R diagrams, and hybrid diagrams, locating waves in genus two R-R diagrams, and the Culling Lemma. Since hyperbolic P/SF knots in \cite{BK20} are described in terms of R-R diagrams, Section~\ref{Heegaard diagrams, R-R diagrams, and Hybrid diagrams} deals with some basics on the three diagrams. Locating waves in genus two R-R diagrams and the Culling Lemma, which are given in Sections \ref{Locating waves in genus two R-R diagrams} and \ref{Culling lists of knot exterior candidates} respectively, are related to find a meridian of a P/SF knot or its exterior $H[R]$, which is one of the steps to find all hyperbolic P/SF knots.

More specifically, Section~\ref{Heegaard diagrams, R-R diagrams, and Hybrid diagrams} deals with some basics on the three diagrams. We combine a Heeagaard diagram and a R-R diagram to make a hybrid diagram, and we make use of it to transform an R-R diagram which has underlying Heegaard diagram with a cut vertex into that with no cut vertex. In Section~\ref{Locating waves in genus two R-R diagrams} we introduce one of the results of \cite{B20}, so-called ``Waves provide meridians", which is originated from \cite{B93}, and consider how to locate waves based at a nonspearating simple closed curve $R$ into an R-R diagram of $R$. When an R-R diagram of $R$ such that $H[R]$ embeds in $S^3$ is given, the location of a wave into the R-R diagram needs to be understood in order to find a meridian of $H[R]$ which can be obtained by surgery on $R$ along a wave due to the result of \cite{B20}.
In Section~\ref{Culling lists of knot exterior candidates}, the ``Culling Lemma'' is established, which is a tool to cull out meridian candidates of $H[R]$.\\

\noindent\textbf{Acknowledgement.} In 2008, in a week-long series of talks to a seminar in the department of mathematics of the University of Texas as Austin, John Berge outlined a project to completely classify and describe the primitive/Seifert knots in $S^3$. The present paper, which provides some of the background materials necessary to carry out the project, is originated from the joint work with John Berge for the project. I should like to express my gratitude to John Berge for his support and collaboration. I would also like to thank Cameron Gordon and John Luecke for their support while I stayed in the University of Texas at Austin.

\section{Heegaard diagrams, R-R diagrams, and Hybrid diagrams}\label{Heegaard diagrams, R-R diagrams, and Hybrid diagrams}

In this section we describe some basics on Heegaard diagrams, R-R diagrams, and hybrid diagrams, which will be used in the other sections.

\subsection{Heegaard diagrams}\hfill
\smallskip

First, we deal with the definition and property of Heegaard diagram of simple closed curves
lying in the boundary of a genus two handlebody, and some of its terminologies that are
needed in this paper.

Suppose $\mathcal{C}$ is a finite set of pairwise disjoint nonparallel simple closed curves in the boundary
of a genus two handlebody $H$, none of the curves bound disks in $H$, and
$\{D_A, D_B\}$ is a complete set of cutting disks of $H$. Cutting $H$
open along $D_A$ and $D_B$ cuts $\mathcal{C}$ into set of arcs
$E(\mathcal{C})$ and cuts $H$ into a 3 ball $W$.
Then $\partial W$ contains disks $D_A^+, D_A^-, D_B^+,$ and $D_B^-$
such that gluing $D_A^+$ to $D_A^-$ and $D_B^+$ to $D_B^-$ reconstitutes $\mathcal{C}$
and $H$. Figures~\ref{ASK-14} and \ref{ASK-145} illustrate this process.

\begin{figure}[tbp]
\centering
\includegraphics[width = 0.6\textwidth]{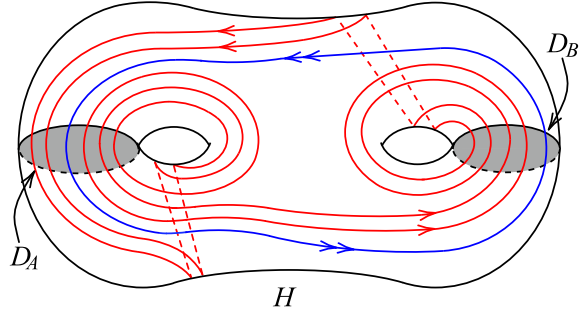}
\caption{Two simple closed curves in the boundary of a genus two handlebody.}
\label{ASK-14}
\end{figure}

\begin{figure}[tbp]
\centering
\includegraphics[width = 0.6\textwidth]{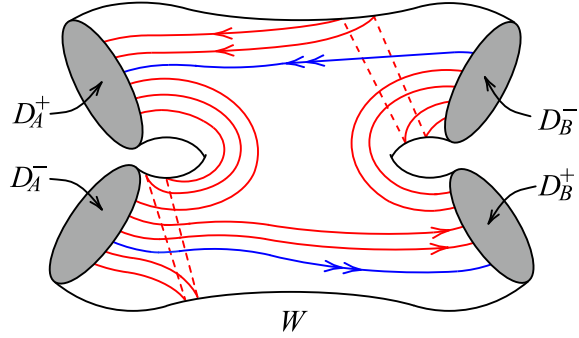}
\caption{A 3-ball $W$ obtained by cutting $H$ open along $D_A$ and $D_B$.}
\label{ASK-145}
\end{figure}

\begin{figure}[tbp]
\centering
\includegraphics[width = 0.7\textwidth]{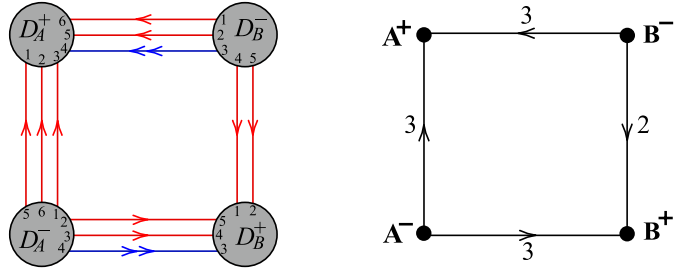}
\caption{Heegaard diagram and its underlying graph.}
\label{ASK-146}
\end{figure}

The set of arcs $E(\mathcal{C})$ forms the edges of \textit{Heegaard diagrams} in
$\partial W$ with ``fat", i.e., disk rather than point, vertices $D_A^+, D_A^-, D_B^+,$
and $D_B^-$. If one ignores how $D_A^+$ and $D_B^+$ are identified with $D_A^-$ and $D_B^-$
to reconstitute $H$, the set of arcs $E(\mathcal{C})$ also forms the edges of a graph
in $\partial W$ with vertices $D_A^+, D_A^-, D_B^+,$ and $D_B^-$. Let $G$ denote
the graph in $\partial W$ whose edges are the arcs in $E(\mathcal{C})$.
Then $G$ is the graph underlying the Heegaard diagram of $\mathcal{C}$.
Some simplification can be made. If there are $n$ parallel edges in $G$,
then we merge these parallel edges into one and recode this edge
by placing the integer $n$ near the edge. Also we smash the disks
$D_A^+, D_B^-, D_B^+,$ and $D_B^-$ in $\partial W$ into points denoted
by $A^+, A^-, B^+,$ and $B^-$. Figure~\ref{ASK-146} shows the Heegaard diagram and the underlying graph of the curves
$k_1$ and $k_2$ in Figure~\ref{ASK-14}.
Note that this graph is not just abstract graph, since it inherits specific
embeddings in the 2-sphere $S^2$ which is homeomorphic to $\partial W$
from the Heegaard diagram which it underlie.

The following lemma, which can be found in \cite{HOT80} or \cite{O79}, shows some possible types of
graphs of Heegaard diagrams of simple closed curves on the boundary of a genus two handlebody.

\begin{figure}[tbp]
\centering
\includegraphics[width = 1.0\textwidth]{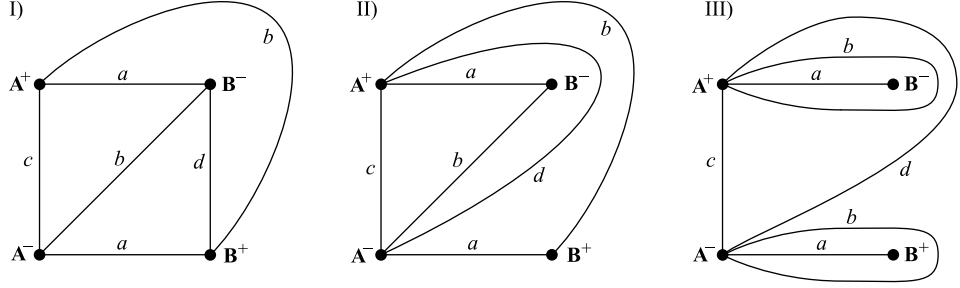}
\caption{The three types of graphs of Heegaard diagrams of simple closed curves on
the boundary of a genus two handlebody $H$ which has cutting disks $D_A$ and $D_B$,
excluding diagrams in which simple closed curves are disjoint from both $\partial D_A$ and $\partial D_B$. }
\label{DPCFig8a-3}
\end{figure}

\begin{lem}
\label{3 types of genus two diagrams}
Let $H$ be a genus two handlebody with a set of cutting disks $\{D_A,D_B\}$ and let $\mathcal{C}$ be a finite set of pairwise disjoint nonparallel simple closed curves on $\partial H$ whose intersections with $\{D_A, D_B\}$ are essential and not both empty. Then, after perhaps relabeling $D_A$ and $D_B$, the Heegaard diagram of $\mathcal{C}$ with respect to $\{D_A, D_B\}$ has the form of one of the three graphs in Figure~\emph{\ref{DPCFig8a-3}}.
\end{lem}

\begin{defn}
[\textbf{Complexity}]
The \emph{complexity} of a Heegaard diagram---without simple closed curve components---is the total number of edges of the diagram. The complexity is denoted by $|\mathcal{C}|$.
\end{defn}

\begin{defn}
[\textbf{cut-vertex}]
If $v$ is a vertex of a connected graph $G$ such that deleting $v$ and
the edges of $G$ meeting $v$ from $G$ disconnects $G$, we say $v$ is a \emph{cut-vertex} of $G$.
\end{defn}

The complexity of the Heegaard diagrams in Figure~\ref{DPCFig8a-3} is $2a+2b+c+d$, and
the Heegaard diagram in Figure~\ref{DPCFig8a-3}c either is not connected or has a cut-vertex.

\begin{defn}
[\textbf{Positive Heegaard Diagram}]
A Heegaard diagram is \emph{positive} if the curves of the diagram can be oriented so that all intersections of curves in the diagram are positive. Otherwise, the diagram is \emph{nonpositive}.
\end{defn}

\noindent See Figures~\ref{DPCFig11i-3} and \ref{DPCFig11h} for examples of a nonpositive diagram and a positive diagram respectively.

\subsection{R-R diagrams}\label{R-R diagrams}\hfill
\smallskip

In this subsection we describe the definition of an R-R diagram and its basic properties.
R-R diagrams are a type of planar diagram related to Heegaard diagrams. These diagrams were originally introduced by Osborne and Stevens in \cite{OS77} and developed by Berge. They are particularly useful for describing embeddings of
simple closed curves in the boundary of a handlebody so that the embedded curves
represent certain conjugacy classes in $\pi_1$ of the handlebody. In addition,
R-R diagrams are much easier to parametrize and to see structural details of systems of curves
than standard Heegaard diagrams.

\subsubsection{Genus $g$ R-R diagrams}\hfill
\label{Genus $g$ R-R diagrams}

\smallskip

We provide the basics of genus $g$ R-R diagrams
describing curves lying in the boundary of a genus $g$ surface.

Consider the following construction of a closed orientable surface $\Sigma_g$ of genus $g$. Start with a planar surface $S_g$ with $g$ boundary components $\beta_i$ and $g$ \emph{handles} $F_i$, $1 \leq i \leq g$, each of which is a once-punctured torus. Then cap off each boundary component of $S_g$ by identifying $\partial F_i$ with $\beta_i$. The resulting surface $\Sigma_g$ is naturally endowed with a set of $g$ pairwise disjoint separating simple closed curves $\beta_i$ for $1 \leq i \leq g$, each of which is the boundary of a corresponding handle $F_i$.

Now suppose $\mathcal{C}$ is a finite collection of pairwise disjoint simple closed curves in $\Sigma_g$. The decomposition of $\Sigma_g$ into a planar component $S_g$ and a set of $g$ handles by the separating curves $\beta_i$, $1 \leq i \leq g$, gives $\Sigma_g$ a structure which can be used to study sets of simple closed curves, such as $\mathcal{C}$, on $\Sigma_g$. Zieschang does this in several places; see e.g. the expository paper \cite{Z88} or \cite{Z62}, \cite{Z63a}, \cite{Z63b}. In Zieschang's terminology, the $\beta_i$s are \emph{belt curves}. It is usually convenient to assume that the number of intersections of the curves in $\mathcal{C}$ with the members of the set of belt curves $\beta_i$ has been minimized by isotopies in $\Sigma_g$. Then a curve $C$ in $\mathcal{C}$ either lies completely in $S_g$, or entirely on one of the $g$ handles $F_i$, or $C$ is cut by its intersections with the belt curves $\beta_i$ into arcs, each either properly embedded and essential in $S_g$, or properly embedded and essential in some handle $F_i$.

\begin{figure}[tbp]
\centering
\includegraphics[width = 0.8\textwidth]{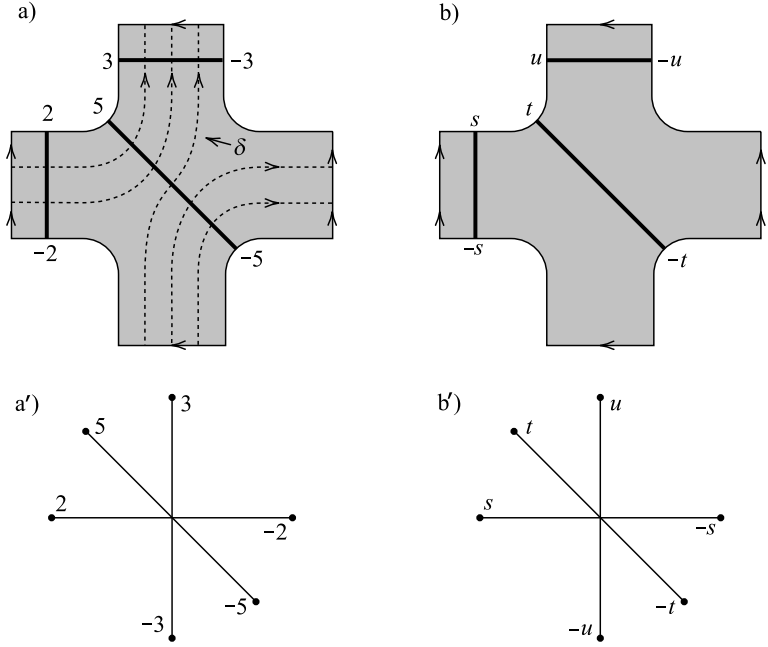}
\caption{Punctured tori handles.}
\label{DPCFig10a}
\end{figure}

A properly embedded essential arc in a handle $F_i$ is a \emph{connection}. Two connections on $F_i$ are \emph{parallel}
if they are isotopic via an isotopy which keeps their endpoints in $\partial F_i$. A collection of pairwise disjoint
connections on a given handle can be partitioned into \emph{bands} of pairwise parallel connections.
Then it is easy to see that, on a once-punctured torus, there can be at most three nonparallel bands of connections.
Figure~\ref{DPCFig10a}a or \ref{DPCFig10a}b illustrates this situation, where the shaded region of Figure~\ref{DPCFig10a}a or \ref{DPCFig10a}b
is an octagon obtained by cutting a punctured torus handle $F$ open along two nonparallel properly embedded essential arcs.
This leaves the four edges of the octagon marked with arrows, which must be identified in pairs to recover the punctured torus.

Note that sets of pairwise nonparallel connections in a once-punctured torus are unique up to homeomorphism. In particular,
if $F$ is a once-punctured torus, $\Delta=\{\alpha_1, \ldots, \alpha_i\}, 1\leq i\leq 3$ is a set of pairwise nonparallel connections
in $F$, and $\Delta'=\{\alpha'_1, \ldots, \alpha'_j\}, 1\leq j\leq 3$ is another set of pairwise nonparallel connections
in $F$, and $i=j$, then there is a homeomorphism of $F$ which takes $\Delta$ to $\Delta'$.

Some simplifications can be made at this point without losing any information about the embedding of the curves of $\mathcal{C}$
in $\Sigma_g$. Suppose $F$ is a handle, and let $\mathcal{C}_{S_g}$ be the set of arcs in which curves of $\mathcal{C}$ intersect
$S_g$. Then each set of parallel connections on $F$ can be merged into a single connection. (This also merges some
endpoints of arcs in $\mathcal{C}_{S_g}$ meeting $\partial F$.)

After such mergers, $F$ carries at most 3 pairwise nonparallel connections. Continuing, after each set of parallel connections
on $F$ has been merged, additional mergers of sets of properly embedded parallel subarcs of $\mathcal{C}_{S_g}$
can also be made; although now whenever, say, $n$ parallel arcs merged into one, this needs to be recorded
by placing the integer $n$, which is called a \textit{weight}, near the single arc resulting from the merger.

Merging parallel connections in each handle turns the set of pairwise disjoint simple closed curves
in $\mathcal{C}$ into a graph $\mathcal{G}$ in $\Sigma_g$ whose vertices are the endpoints
of the remaining connections in each handle. Clearly $\mathcal{G}$ and its embedding in $\Sigma_g$
completely encodes the embedding of the curves of $\mathcal{C}$ in $\Sigma_g$.

A problem with $\mathcal{G}$ is that in general it is not planar. However, $\mathcal{G}$ always
has a nice immersion $I(\mathcal{G})$ in the plane $\mathbb{R}^2$, which becomes an R-R diagram
of the curves of $\mathcal{C}$ in $\Sigma_g$.

To produce $I(\mathcal{G})$, first remove a small disk $D$, disjoint from $\mathcal{G}$, from
the interior of $S_g$. Then embed $S_g-D$ in $\mathbb{R}^2$ so that each handle $F$ bound
disjoint round disks, say $\mathcal{F}$ in $\mathbb{R}^2$.

Next, note that if $\alpha$ and $\alpha'$ are nonparallel connections on a handle $F$,
the endpoints of $\alpha$ separate the endpoints of $\alpha'$ in the belt curve $\partial F$.
It follows that, if $u$ and $v$ in $\partial F$ are the endpoints of a connection
$\alpha$ in $\partial F$, we may assume that $u$ and $v$ bound a diameter $d_\alpha$ in
$\mathcal{F}$. This results in each round disk $\mathcal{F}$ containing
$0, 1, 2,$ or $3$ diameters passing through its center, with each diameter an image
of a connection in $F$, where the number of such diameters depends upon whether $F$ originally
contained respectively $0, 1, 2,$ or $3$ bands of parallel connections. Figure~\ref{DPCFig10a}
shows 3 connections and thus 3 diameters.

Now we encode the endpoints of each diameter, i.e., endpoints of each connection in $F$.
Representatives of three types of connections on $F$ are shown in bold in Figure~\ref{DPCFig10a}a,
along with a nonseparating simple closed curve $\delta$ lying in the interior of $F$,
which appears as the ``dotted'' curve in Figure~\ref{DPCFig10a}a. Then the integer labels
at the ends of the three bands of connections in Figure~\ref{DPCFig10a}a give signed
intersection numbers of any oriented connections lying in the three illustrated bands of connections in $F$ with $\delta$.

Figure~\ref{DPCFig10a}b again shows $F$, but the set of integers, which labeled the ends of
the bands of connections in Figure~\ref{DPCFig10a}a, has been replaced with a set of integer
parameters $\{s,t,u,-s,-t,-u\}$, which appear in clockwise cyclic order $(s,t,u,-s,-t,-u)$
around the boundary $\beta$ of $F$, and the curve $\delta$ has disappeared.
Figures~\ref{DPCFig10a}a$'$ and \ref{DPCFig10a}b$'$ illustrate the labels of the endpoints of corresponding diameters in $F$.
A connection whose endpoints are labeled by $x$ and $-x$ is simply said to be $x$-connection or $-x$-connection.

Now the well-known characterization of the homology classes in
$H_1(F; \mathbb{Z})$ which can be represented by simple closed curves implies the following proposition.

\begin{prop}
There exists a nonseparating simple closed curve $\delta$ in $F$ such that $\delta$ has the algebraic intersection numbers shown in Figure \emph{\ref{DPCFig10a}b} with oriented connections in the bands of connections of Figure \emph{\ref{DPCFig10a}b} if and only if $\gcd(s,u) = 1$ and $t = s + u$.
Furthermore, the isotopy class of $\delta$ in $F$ is uniquely determined by the algebraic intersection numbers of any two nonparallel bands of connections in $F$.
\end{prop}

Note that if there are two or three bands of connections in $F$ and if the absolute value of one of labels of endpoints of the bands of connections is
greater than $2$, then the labels determine the connections up to a homeomorphism of $F$ which is the identity on $\partial F$. If
all the labels have absolute value $1$ and $2$, or $1$ only, the labels do not determine the connections in $F$.
Furthermore, if there is only one band of connections with label $s>0$, there are $\phi(s)$ inequivalent connections with the same label $s$, where $\phi(s)$
is the number of positive integers $s'$ such that gcd$(s, s')=1$. This is what we want to generalize in the following way.

A generalization occurs when the punctured torus $F$ has been endowed with a pair of embedded oriented simple closed curves, say, $\vec{m}$ and $\vec{l}$, which meet transversely at a single point and form a basis for $H_1(F)$. Then each band of connections on $F$ can be labeled with an ordered pair of integers which represent the coordinates of an oriented connection in that band with respect to the basis formed by $\vec{m}$ and $\vec{l}$. This ordered pair of integers
characterizes the connection up to a homeomorphism of $F$. Furthermore if there are the pairs of labels on distinct bands, then they must satisfy a unimodular determinant condition. That is, if two distinct bands of connections have labels of the form $(s,s')$ and $(u,u')$, then $ su'-us' = \pm 1$. Proposition~\ref{R-R labels and unimodular determinant conditions} and Figure~\ref{DPCFig11d} illustrate the situation.

\begin{prop}
\label{R-R labels and unimodular determinant conditions}
Suppose $\alpha$ and $\beta$ are simple closed curves with an R-R diagram of the form shown in Figure~\emph{\ref{DPCFig11d}a}. If the diagram of Figure~\emph{\ref{DPCFig11d}a} is interpreted as describing embeddings of $\alpha$ and $\beta$ in the punctured torus $T$ shown in Figure~\emph{\ref{DPCFig11d}b} so that $[ \vec{\alpha}] = s[\vec{l}\:] +s'[\vec{m}]$ and $[\vec{\beta}\,] = u[\vec{l}\:] +u'[\vec{m}]$ in $H_1(F)$, then $su'-us' = 1$.
\end{prop}

\begin{proof}
If $\alpha$ and $\beta$ embed in $T$ with $\alpha \cap \beta$ a single point of transverse intersection, as shown in Figure~\ref{DPCFig11d}a, then there is an orientation preserving homeomorphism $\theta$ of $F$ such that $\theta(\vec{l}\:) = \vec{\alpha}$ and $\theta(\vec{m}) = \vec{\beta}$.
If $[ \vec{\alpha}] = s[\vec{l}\:] +s'[\vec{m}]$ and $[\vec{\beta}\,] = u[\vec{l}\:] +u'[\vec{m}]$ in $H_1(F)$, then $\theta$ induces an automorphism of $H_1(F)$ whose determinant is equal to $su'-us'$, which must be plus one, since $\theta$ is orientation-preserving.
\end{proof}

\begin{figure}[tbp]
\centering
\includegraphics[width = 0.7\textwidth]{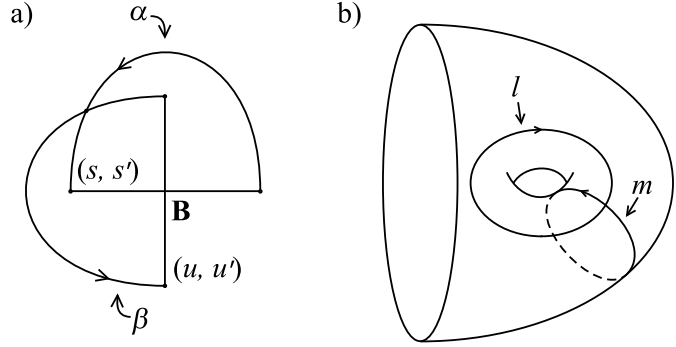}
\caption{A pair of figures, used to show that labels $(s,s')$ and $(u,u')$  of connections on a handle $B$ of an R-R diagram, as shown in Figure~\ref{DPCFig11d}a, must satisfy $su'-s'u = 1$, if $[ \vec{\alpha}] = s[\vec{l}] +s'[\vec{m}]$ and $[\vec{\beta}] = u[\vec{l}] +u'[\vec{m}]$ in $H_1$ of the punctured torus in Figure~\ref{DPCFig11d}b.}
\label{DPCFig11d}
\end{figure}

One practical example of this generalization can be made when we consider knots lying in a genus $g$ Heegaard surface $\Sigma_g$ of $S^3$
which is standardly embedded and bounds two handlebodies $H$ and $H'$ in $S^3$. As constructed at the beginning, suppose $\Sigma_g$ is obtained
from a planar surface $S_g$ with $g$ boundary components $\beta_i$ by gluing $g$ handles $F_i$ along $\beta_i$, $1\leq i\leq g$.
Also suppose for each $i$, $D_{A_i}$ and $D_{A'_i}$ are cutting disks of $H$ and $H'$ respectively such that $\partial D_{A_i}$ and $\partial D_{A'_i}$
lie in a handle $F_i$ and intersect transversely at a point. Then the two curves $\partial D_{A_i}$ and $\partial D_{A'_i}$
give the labelling of one end of a band of connections by $(s, s')$ in $F_i$. The R-R diagram that is constructed in this way
provides an embedding of simple closed curves(or knots) lying in a Heegaard surface of $S^3$ which is standardly embedded in $S^3$. Thus
we make the following terminology.

\begin{defn}
[\textbf{GDS R-R diagram}] \label{GDS R-R diagram}
Let $\Sigma_g$ be a genus $g$ Heegaard surface of $S^3$ bounding two handlebodies $H$ and $H'$ and $\mathcal{C}$ be the set of pairwise disjoint simple closed curves in $\Sigma_g$. Suppose $\Sigma_g$ is obtained from a planar surface $S_g$ with $g$ boundary components $\beta_i$ by gluing $g$ handles $F_i$ along $\beta_i$, $1\leq i\leq g$.
Suppose $\mathfrak{D}=\{D_{A_1}, \ldots, D_{A_g}\}$ and $\mathfrak{D}'=\{D_{A'_1}, \ldots, D_{A'_g}\}$ are complete sets of cutting disks of $H$ and $H'$ respectively such that $\partial D_{A_i}$ and $\partial D_{A'_i}$, $1\leq i\leq g$, lie in a handle $F_i$ and intersect transversely at a point.
Then the R-R diagram of $\mathcal{C}$ on $\Sigma_g$ with respect to $\mathfrak{D}$ and $\mathfrak{D}'$ is said to be a \textit{great disk system} of R-R diagram of $\mathcal{C}$, or simply \textit{GDS R-R diagram} of $\mathcal{C}$. Also each handle $F_i$ in the R-R diagram is said to be an $(A_i, A'_i)$-handle.
\end{defn}

One example of an R-R diagram and a GDS R-R diagram will be given in the next subsection.

\subsubsection{Genus two R-R diagrams}\hfill
\label{Genus two R-R diagrams}

\smallskip

Since we are interested in curves in the boundary of a genus two handlebody,
as a special case, this subsection shows how to make a geuns two R-R diagram, or simply R-R diagram of, for example, two curves $k_1$ and $k_2$ in Figure~\ref{ASK-14}. We follow the procedure given in the previous subsection. Note that the basics of genus two R-R diagrams were also given in \cite{B09} and an example of genus two R-R diagrams was given in \cite{K14}, which we will use here as a same example.

By considering two parallel separating curves, i.e., belt curves
as shown in Figure~\ref{ASK-141}, we decompose the boundary of
$H$ into two handles $F_A$, $F_B$, and one annulus $\mathcal{A}$, so that
the two handles $F_A$ and $F_B$ contain $\partial D_A$ and $\partial D_B$ respectively,
which can be considered as a nonseparating curve $\delta$ in each handle.
Figure~\ref{ASK-142} shows this decomposition.

Note from Figure~\ref{ASK-142} that there are three nonparallel bands of connections in $F_A$, each of which
consists of one connection, and there are two
nonparallel bands of connections in $F_B$, one of which contains one connection and the other contains two.
Therefore there are three diameters in $\mathcal{F}_A$ and two in $\mathcal{F}_B$ as shown in
Figure~\ref{ASK-143}.

\begin{figure}[t]
\centering
\includegraphics[width = 0.6\textwidth]{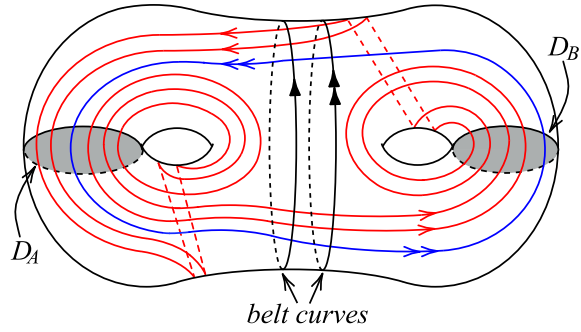}
\caption{Belt curves bounding an annulus in a genus two surface.}
\label{ASK-141}
\end{figure}

\begin{figure}[t]
\centering
\includegraphics[width = 0.7\textwidth]{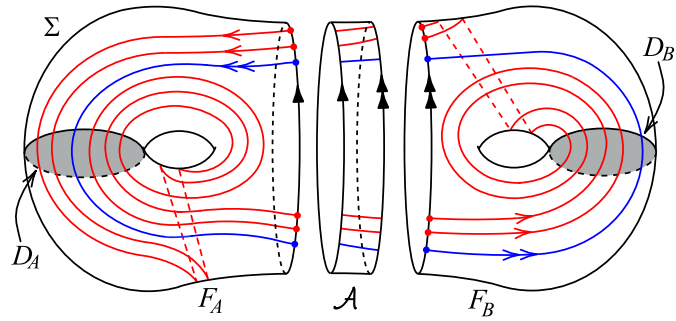}
\caption{An annulus $\mathcal{A}$ and two handles $F_A$ and $F_B$.}
\label{ASK-142}
\end{figure}

\begin{figure}[t]
\centering
\includegraphics[width = 0.7\textwidth]{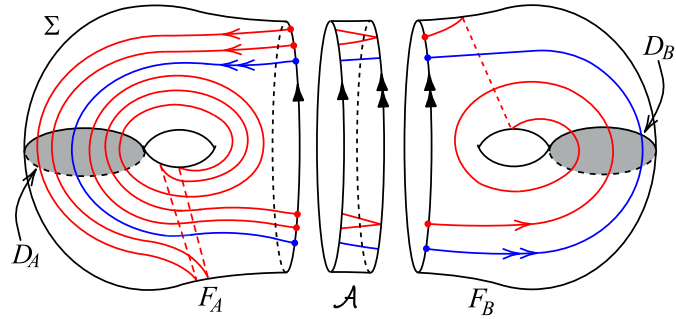}
\caption{Merging the parallel connections and the endpoints of the arcs in the annulus $\mathcal{A}$.}
\label{ASK-1421}
\end{figure}

\begin{figure}[t]
\centering
\includegraphics[width = 0.6\textwidth]{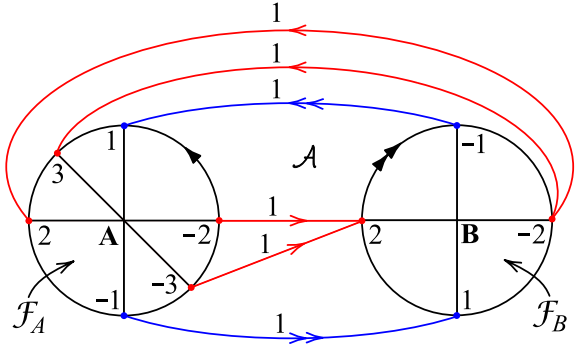}
\caption{Immersion of curves of $\mathcal{C}$ into $S^2$ which becomes an corresponding R-R diagram. }
\label{ASK-143}
\end{figure}

\begin{figure}[t]
\centering
\includegraphics[width = 0.6\textwidth]{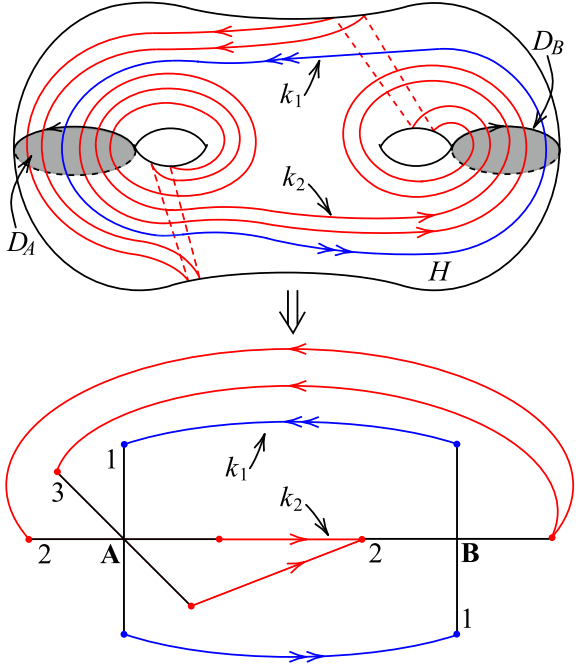}
\caption{Transformation into R-R diagram.}
\label{ASK-144}
\end{figure}

Merging the parallel connections and thus the endpoints of the connections
leads to mergers of the endpoints of arcs in the annulus $\mathcal{A}$ as in Figure~\ref{ASK-1421}.
We observe from Figure~\ref{ASK-1421} that after this merger, no pair of the arcs in $\mathcal{A}$ have the
same endpoints and thus each arc has label $1$.

In order to put the labels of the endpoints of each diameter (or each band of connections),
we take the $\partial D_A$ and $\partial D_B$ in $F_A$ and $F_B$ respectively
as a nonseparating curve $\delta$ in a handle.
Then the integer parameters $s,t,u$ of the diameters on $F_A$(or $F_B$)
imply the intersection numbers with the cutting disk $D_A$(or $D_B$). The labels
of each diameter are given in Figure~\ref{ASK-143}.

With all of the information obtained, we can make an immersion of $k_1$ and $k_2$ into $S^2$ which becomes a corresponding R-R diagram.
First embed the annulus $\mathcal{A}$ in $S^2$ obtained by deleting
two disks from $S^2$ whose boundaries correspond to $\partial F_A$ and $\partial F_B$
as shown in Figure~\ref{ASK-143}. We put the capital letters \textbf{A} and \textbf{B} to indicate
correspondence to the two handles $F_A$ and $F_B$ respectively and we call the corresponding handles as $A$-handle and $B$-handle.
Last, we disregard the boundary circles of $F_A$ and $F_B$ in Figure~\ref{ASK-143} and one of the labels of the endpoints of each band of connections to obtain the corresponding
R-R diagram. Figure~\ref{ASK-144} shows the curves $k_1$ and $k_2$ in $\partial H$ and the corresponding R-R diagram.

Since a complete set of cutting disks $\{D_A, D_B\}$ of $H$ is used in the construction of an R-R diagram, there is a Heegaard diagram \textit{underlying} the R-R diagram, which is obtained by cutting $H$ open along the cutting disks $D_A$ and $D_B$.

R-R diagrams give sufficient information about conjugacy classes of the element represented by a simple closed curve $k$ in $\pi_1(H)$.
$\pi_1(H)$ is a free group $F(A,B)$ which is generated by $A$ and $B$ dual to the cutting disks $D_A$ and $D_B$ respectively.
It follows from Figure~\ref{ASK-144} that $k_1$ and $k_2$ represent the conjugacy classes of $AB$ and $A^3B^2A^2B^2$ respectively in $\pi_1(H)$.

\begin{figure}[tbp]
\centering
\includegraphics[width = 0.6\textwidth]{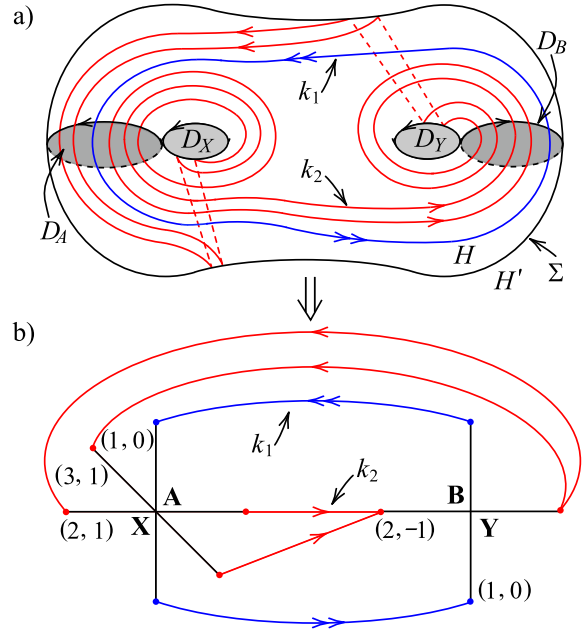}
\caption{Two simple closed curves $k_1$ and $k_2$ lying in a standard genus two Heegaard splitting $(\Sigma, H, H')$ of $S^3$ and their GDS R-R diagrams.}
\label{ASK-147}
\end{figure}

If we consider the handlebody $H$ to be embedded standardly in $S^3$ as in Figure~\ref{ASK-147}a so that the complement of $H$ is another
handlebody $H'$ with a complete set of cutting disks $\{D_X, D_Y\}$ and $\Sigma$ is the common boundary of $H$ and $H'$.
Figure~\ref{ASK-147}b shows a GDS R-R diagram of $k_1$ and $k_2$ with respect to the Heegaard splitting $(\Sigma; H, H')$ with complete sets of cutting disks $\{D_A, D_B\}$ and $\{D_X, D_Y\}$.
Note that $k_1$ has $(1,0)$-connections in both $A$- and $B$-handles, and $k_2$
has $(3, 1)$- and $(2,1)$-connection in the $(A, X)$-handle, and $(2, -1)$-connection in the $(B, Y)$-handle. We see that the labels of the connections satisfy the unimodular
determinant condition in Proposition~\ref{R-R labels and unimodular determinant conditions}.

Now we finish this subsection by defining the equivalence of R-R diagrams.

\begin{defn}
[\textbf{Equivalent R-R diagrams}]\label{Equivalent R-R diagrams}
Let $k$ and $k'$ be two simple closed curves on the boundary of a genus two handlebody $H$. Let $\mathcal{D}_k(\mathcal{D}_{k'}, \textrm{resp}.)$
be an R-R diagram of $k(k', \textrm{resp}.)$ underlying a Heegaard diagram $\mathbb{D}_k(\mathbb{D}_{k'}, \textrm{resp}.)$ of $k(k',$ $\textrm{resp}.)$
with respect to a complete set of cutting disks $\{D_A, D_B\}(\{D_{A'}, D_{B'}\}, \textrm{resp})$.

Then $\mathcal{D}_k$ is \textit{equivalent} to $\mathcal{D}_{k'}$ if there is a homeomorphism of $H$ onto itself sending $\{D_A, D_B\}$ to $\{D_{A'}, D_{B'}\}$ and
sending $k$ to $k'$.
\end{defn}

\smallskip

\subsection{Hybrid diagrams}\hfill
\label{hybrid diagrams}

\smallskip

In this section, we will explain the definition of hybrid diagrams (of genus two)
and its basic properties. Hybrid diagrams are a type of planar
graph which is a combination of Heegaard diagrams and R-R diagrams.

Some R-R diagrams of simple closed curves might be needed to transform
into other R-R diagrams to apply for appropriate theories related to R-R diagrams.
For example, some R-R diagrams might have underlying Heegaard diagrams with a cut-vertex,
which happen when all of the labels of connections in one handle are $0$ or $1$.
In this case we have some difficulty in using theories such as surgery of a curve along a wave described in Theorem~\ref{waves provide meridians},
which can only be applied to Heegaard diagrams
with no cut-vertex. Therefore we need to transform R-R diagram into
that which overlies a Heegaard diagram with no cut-vertex. This transformation can be achieved by a change of cutting disks
of a genus two handlebody $H$ which induces an automorphism on the free group of $\pi_1(H)$. Then hybrid diagrams are useful to perform such a change of cutting disks.

To define hybrid diagrams, we start with an R-R diagram of a simple closed curve $k$. By the construction of R-R diagram from a genus two handlebody $H$ with a complete set
of cutting disks $\{D_A, D_B\}$, there is a separating curve $\beta$ (one of the belt curves) in $\partial H$ splitting $\partial H$ into two handles $F_A$
and $F_B$ such that $F_A$ ($F_B$, respectively) contains $\partial D_A$ ($\partial D_B$, respectively).
Now we consider the Heegaard diagram underlying the R-R diagram, which is obtained by
cutting $H$ open along the cutting disks $D_A$ and $D_B$. The separating curve $\beta$ also separates this diagram into two parts $F'_A$
and $F'_B$ such that $F'_A$ ($F'_B$, respectively) contains $\partial D_A$ ($\partial D_B$, respectively).
Figures~\ref{PPower12}a and \ref{PPower12}b show an example of R-R diagram of a simple closed curve $k$ and its underlying Heegaard diagram with the separating curve $\beta$ in each of the diagrams. Here, we assume that $S,T>0$ and $|S-T|>1$ in Figure~\ref{PPower12}a.
Note that since the labels of the endpoints of the bands of connections in the $A$-handle are only $0$ or $1$, the underlying Heegaard diagram has a cut-vertex.

To get a Heegaard diagram without no cut-vertex, we need to perform a change of cutting disks as follows.
First, replace one of two handles, say, the $A$-handle corresponding to $F_A$ in the R-R diagram by $F'_A$ in the Heegaard diagram.
In the example of the R-R diagram in Figure~\ref{PPower12}a, since a cut-vertex arises due to the $A$-handle where there are only $0$ and $1$-connections,
we ought to replace the $A$-handle in order to obtain a Heegaard diagram with no cut-vertex.
The resulting diagram is shown in Figure~\ref{PPower13}a and is called a \textit{hybrid diagram}.

\begin{figure}[tbh]
\centering
\includegraphics[width = 1\textwidth]{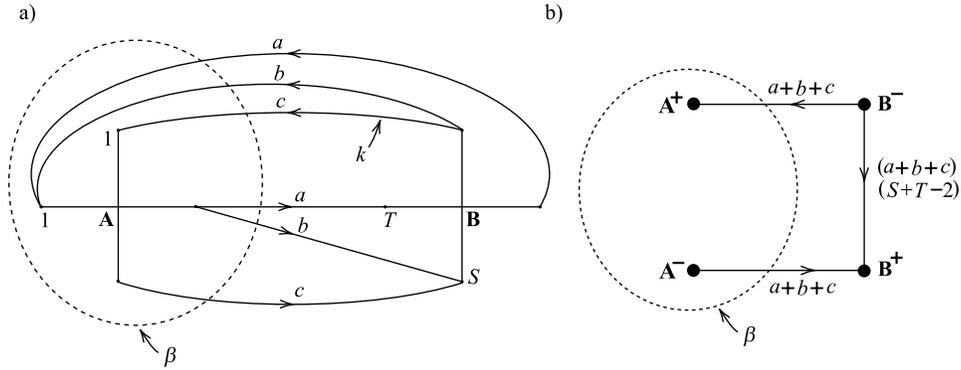}
\caption{An R-R diagram and its underlying Heegaard diagram which has a cut-vertex.}
\label{PPower12}
\end{figure}

Using a hybrid diagram, we can perform a change of cutting disks of a genus two handlebody $H$ with a complete set
of cutting disks $\{D_A, D_B\}$ to obtain a new R-R diagram.
Suppose, for example, a simple closed curve $k$ has a hybrid diagram of the form shown in Figure~\ref{PPower13}a. Now we drag the vertex $A^-$
together with the edges of $k$ over the $S$-connection on the $B$-handle. This performance
corresponds to a change of cutting disks inducing an automorphism of $\pi_1(H)$ which takes $A\mapsto AB^{-S}$ and leaves $B$ fixed.
In other words, with the cutting disk $D_A$ fixed, we change the cutting disk $D_B$ into the cutting disk $D_{B'}$ which
is obtained by bandsumming $D_B$ with $D_A$ along the arcs of $k$ $S$ times. This induces an automorphism of $\pi_1(H)$ which takes $A\mapsto AB^{-S}$ and leaves $B$ fixed.
Then the resulting hybrid diagram is shown in Figure~\ref{PPower13}b,
where the $B$-handle has only one connection labeled by $T-S$. Retrieving the $A$-handle from the resulting hybrid diagram, i.e., the number of
bands of connections in the $A$-handle depends on how the two copies $A^+$ and $A^-$ are identified.
However, since there are $b+c$ edges connecting the vertices $A^+$ and $A^{-}$, the $A$-handle has a band of connection whose label is greater than 1 in the resulting R-R diagram, which implies that the underlying Heegaard diagram has no cut-vertex.

\begin{figure}[t]
\centering
\includegraphics[width = 1\textwidth]{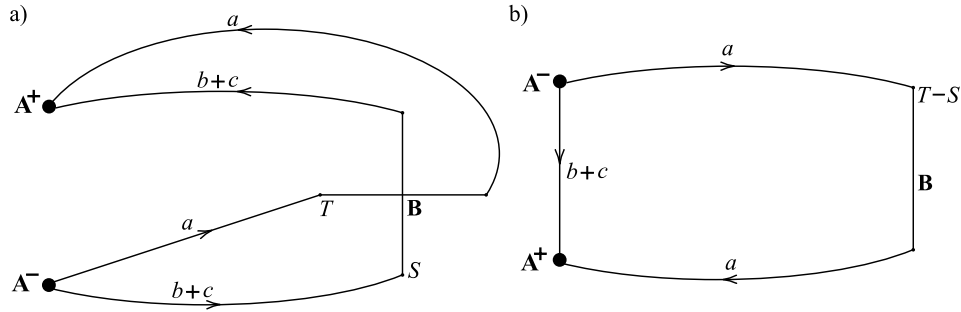}
\caption{Change of cutting disks of a genus two handlebody using a hybrid diagram inducing an automorphism of $\pi_1(H)$ which takes $A\mapsto AB^{-S}$ and leaves $B$ fixed.}
\label{PPower13}
\end{figure}

\section{Locating waves in genus two R-R diagrams}\label{Locating waves in genus two R-R diagrams}

Suppose $R$ is a nonseparating simple closed curve on the boundary of a genus two handlebody $H$ such that $H[R]$ embeds in $S^3$, i.e., $H[R]$ is an exterior of a knot $k$ in $S^3$.
It is shown in \cite{B20} that a meridian of $H[R]$ (or $k$) can be obtained from $R$ by surgery along a wave based at $R$. Recall that a wave on the curve $R$ in $\partial H$ is an arc $\omega$ whose endpoints lies on $R$ with the opposite signs. The following is one of the results of \cite{B20}, which shows how to get a meridian of $H[R]$.

\begin{thm}
[\textbf{Waves provide meridians}]
\label{waves provide meridians}
Let $H$ be a genus two handlebody with a set of cutting disks $\{D_A,D_B\}$ and let $R$ be a nonseparating simple closed curve on $\partial H$ such that the Heegaard diagram $\mathbb{D}_R$ of $R$ with respect to $\{D_A,D_B\}$ is connected and has no cut-vertex. Suppose, in addition, that the manifold $H[R]$ embeds in $S^3$. Then $\mathbb{D}_R$ determines a wave $\omega$, based at $R$, such that if $m$ is a boundary component of a regular neighborhood of $ R \cup \omega $ in $\partial H$, with $m$ chosen so that it is not isotopic to $R$, then $m$ represents the meridian of $H[R]$. Furthermore, the wave $\omega$ determined by $R$ can be obtained as follows:

\begin{enumerate}
\item If $\mathbb{D}_R$ is nonpositive, then $\omega$ is a vertical wave $\omega_v$ which is isotopic to a subarc of the boundary of one of $D_A$ and $D_B$ with which $R$ has both positive and negative signed intersections.
\item If $\mathbb{D}_R$ is positive, then $\omega$ is a horizontal wave $\omega_h$ such that one endpoint of $\omega_h$ lies on an edge of $\mathbb{D}_R$ connecting vertices $A^+$ and $A^-$, while the other endpoint of $\omega_h$ lies on an edge of $\mathbb{D}_R$ connecting vertices $B^+$ and $B^-$.
\end{enumerate}
\end{thm}

\begin{figure}[tbp]
\centering
\includegraphics[width = 1\textwidth] {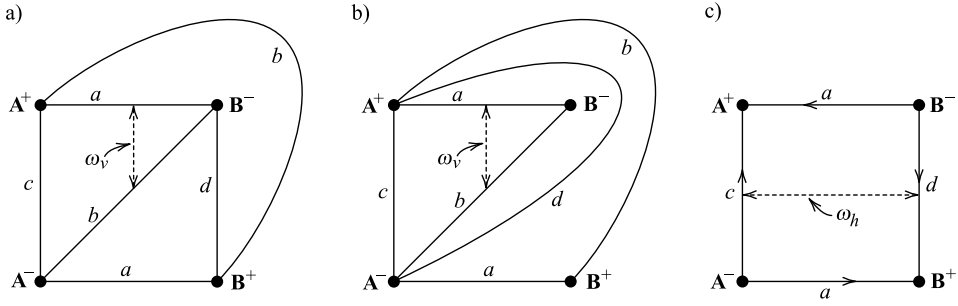}
\caption {A vertical wave $\omega_v$ in a nonpositive Heegaard diagram in a) and b) where $R$ has both positive and negative signed intersections with $D_B$, and a horizontal wave $\omega_h$ in a positive Heegaard diagram in c). They are said to be distinguished
in the sense that they can be used in a surgery on $R$ to obtain a meridian of $H[R]$.}
\label{DPCFig8a-4}
\end{figure}

Figures~\ref{DPCFig8a-4}a and \ref{DPCFig8a-4}b show vertical waves $\omega_v$ when
$R$ has both positive and negative signed intersections with the cutting disk $D_B$ so
that the Heegaard diagram $\mathbb{D}_R$ is nonpositive. Figure~\ref{DPCFig8a-4}c shows a horizon wave $\omega_h$ when $\mathbb{D}_R$ is positive. Vertical waves and horizontal waves which are used to find a representative of a meridian of $H[R]$ as described in Theorem~\ref{waves provide meridians} are said to be \textit{distinguished}.

Using Theorem~\ref{waves provide meridians}, we will be able to rule out some R-R diagrams of simple closed curves $R$ in which $H[R]$ cannot embed in $S^3$, prove the main results of this paper, and characterize R-R diagrams of a torus or cable knot. However
as a prerequisite for this we need to figure out how to locate distinguished waves from Heegaard diagram of a simple closed curve $R$ into its R-R diagram. The rest of this section covers this.

Suppose $\mathcal{D}_R$ is a genus two R-R diagram of a simple closed curve $R$ such that the Heegaard diagram $\mathbb{D}_R$ underlying $\mathcal{D}_R$ has a graph $G_R$ which is connected and has no cut-vertex.

If $\mathcal{D}_R$ is nonpositive, a vertical wave $\omega_v$ is easy to locate in $\mathcal{D}_R$ because it is isotopic to a subarc of the boundary of one of $D_A$ and $D_B$ with which $R$ has both positive and negative signed intersections. For example,
if $R$ appears on $\partial H$ as illustrated in Figure~\ref{DPCFig11i-3}a, then the
Heegaard diagram $\mathbb{D}_R$ is nonpositive such that $R$ has both positive
and negative signed intersections with $\partial D_A$. Therefore $\omega_v$
is isotopic to a subarc of $\partial D_A$ and appears in the Heegaard diagram $\mathbb{D}_R$ and in the R-R diagram $\mathcal{D}_R$ as shown in Figures~\ref{DPCFig11i-3}b and \ref{DPCFig11i-3}c respectively.

\begin{figure}[tbp]
\centering
\includegraphics[width = 0.95\textwidth]{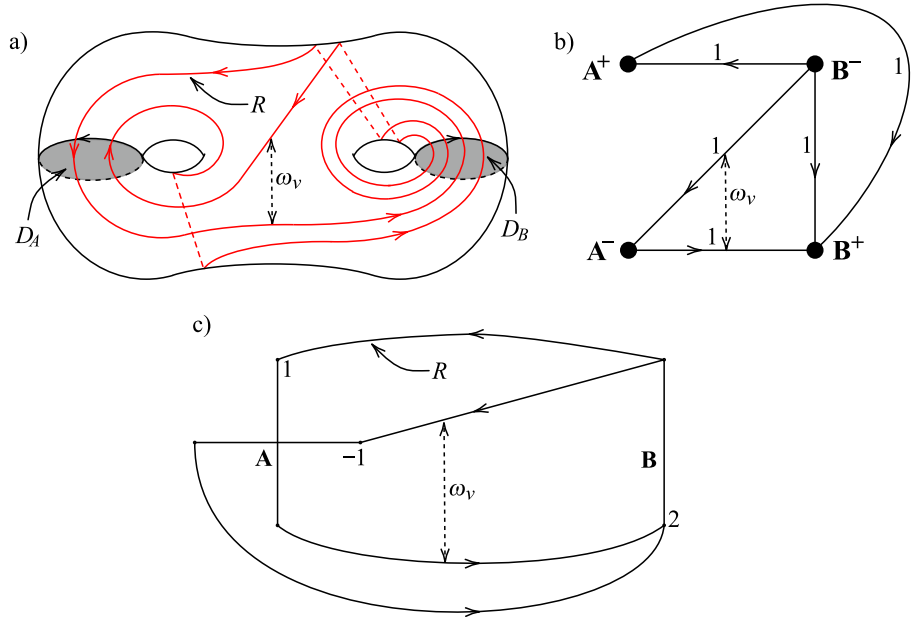}
\caption {A vertical wave $\omega_v$ in a nonpositive Heegaard diagram $\mathbb{D}_R$ and in the R-R diagram $\mathcal{D}_R$.}
\label{DPCFig11i-3}
\end{figure}

If $\mathcal{D}_R$ is positive, a distinguished wave $\omega$ is a  horizontal wave $\omega_h$ which may require some care, because the locations of the two points $\omega_h \cap R$ on $R$ may not be immediately apparent in the R-R diagram $\mathcal{D}_R$.

The next proposition and Figures \ref{DPCFig10d} -- \ref{DPCFig9h}, which follow, address this question by indicating how to locate the endpoints of horizontal waves in R-R diagrams.

\begin{figure}[tbp]
\centering
\includegraphics[width = 0.85\textwidth]{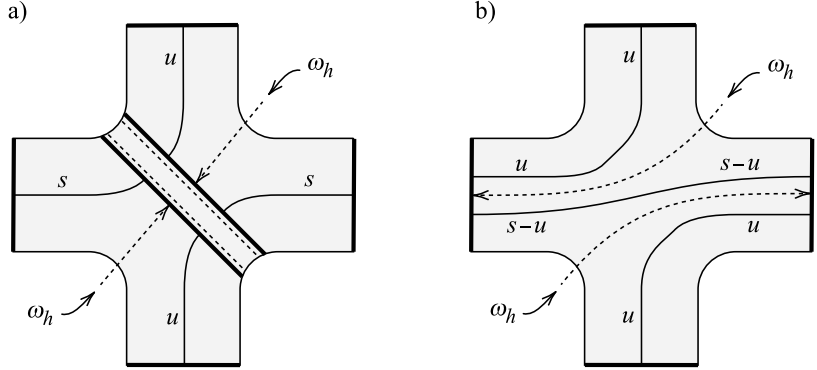}
\caption{Two figures showing a once-punctured torus $T$ in a genus two Heegaard surface in which $T$ corresponds to a handle of an R-R diagram whose underlying Heegaard diagram has a horizontal wave $\omega_h$. Figure \ref{DPCFig10d}a illustrates a situation in which there are three nonempty bands of connections in $T$ meeting the boundary of the cutting disk $\Delta$ lying in $T$, and these bands meet $\partial \Delta$ $s$, $u$ and $s + u$ times with $s,u > 0$. Figure \ref{DPCFig10d}b illustrates a situation in which there are only two nonempty bands of connections in $T$ and $s > u > 0$.
Note that, in each case, $\omega_h$ has an endpoint on a connection in $T$ which borders the band of connections in $T$ that meet $\partial \Delta$ maximally.}
\label{DPCFig10d}
\end{figure}

\begin{prop}
[\textbf{Horizontal Waves Terminate at Connections Bordering Band of Connections With Maximal Labels}]\label{Locate horizontal waves}
Suppose $\mathcal{D}_R$ is a genus two R-R diagram of a simple closed curve $R$ such that the Heegaard diagram $\mathbb{D}_R$ underlying $\mathcal{D}_R$ is positive,  connected, and has no cut-vertex.

Suppose $T$ is a once-punctured torus handle of $\mathcal{D}_R$, and $\omega_h$ is a horizontal wave which meets $R$ in $T$. Let $\Delta$ be the cutting disk of the underlying handlebody whose boundary $\partial \Delta$ lies in $T$.
Then, in each case, $\omega_h$ has an endpoint on a connection in $T$ which borders the band of connections with maximal label, i.e., which meet $\partial \Delta$ maximally.
\end{prop}

\begin{proof}
See Figures \ref{DPCFig10d}a and \ref{DPCFig10d}b, which illustrate the situation.
\end{proof}

\begin{figure}[tbp]
\includegraphics[width = 1\textwidth]{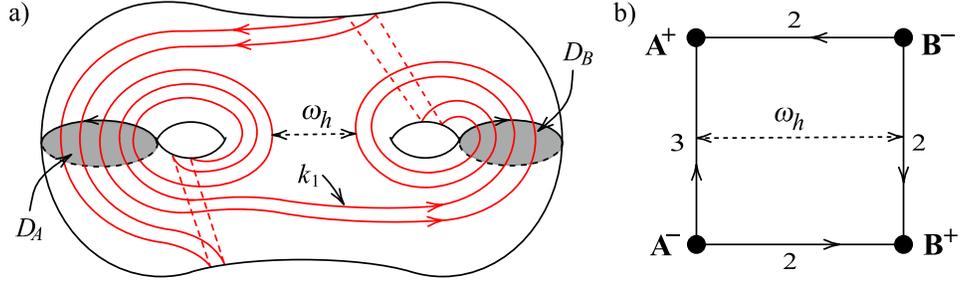}\caption{The horizontal wave $\omega_h$ based at $k_1$ in a genus two surface and in a Heegaard diagram}\label{DPCFig11h}
\end{figure}

For example, consider a simple closed curve $k_1$ in the boundary of a genus two
handlebody $H$ in Figure~\ref{DPCFig11h}a, which is described in Section~\ref{R-R diagrams}. Then its Heegaard diagram $\mathbb{D}_{k_1}$ shown in Figure~\ref{DPCFig11h}b is positive, connected, and has no cut-vertex and thus
it has a horizontal wave $\omega_h$ as in Figure~\ref{DPCFig11h}b. In $\partial H$, $\omega_h$ appears as in Figure~\ref{DPCFig11h}a.
Then we can observe that as Proposition~\ref{Locate horizontal waves} indicates, in $\partial H$ $\omega_h$ terminates at the connections of $k_1$ with maximal intersection number $3$ in the $A$-handle and $2$ in the $B$-handle.

Figures \ref{DPCFig9a}, \ref{DPCFig9b}, \ref{DPCFig9c} and \ref{DPCFig9d} illustrate how an isotopy, which detours the band of connections with maximal labels on a handle, can be used to move an endpoint of a horizontal wave into the annulus between the handles of an R-R diagram.

\begin{figure}[tbp]
\centering
\includegraphics[width = 0.8\textwidth]{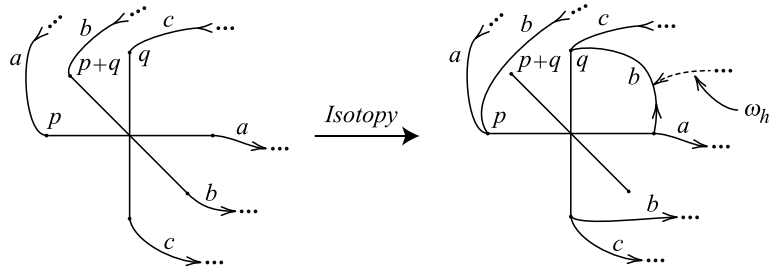}
\caption{Locating the attachment point of a horizontal wave $\omega_h$, meeting an oriented curve from the right, on a handle with three nonempty bands. Here $a, b, c, > 0$, $p, q > 0$ and $\gcd(p,q) = 1$.}
\label{DPCFig9a}
\end{figure}

\begin{figure}[tbp]
\centering
\includegraphics[width = 0.8\textwidth]{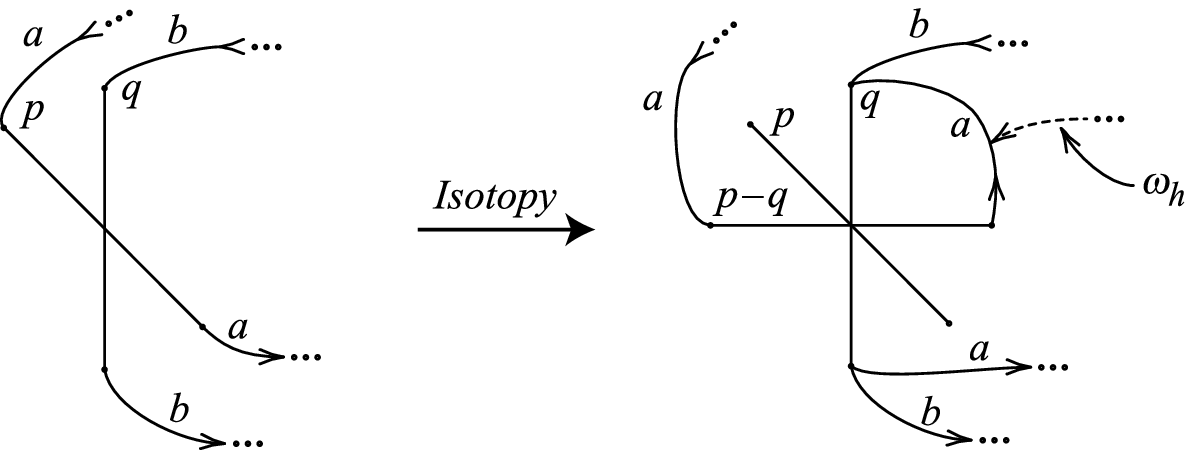}
\caption{Locating the attachment point of a horizontal wave $\omega_h$, meeting an oriented curve from the right, on a handle with two nonempty bands when $p > q > 0$, $a, b > 0$, and $\gcd(p,q) = 1$. (Note the difference between this figure and Figure \ref{DPCFig9c}, which shows the situation when $q > p > 0$.)}
\label{DPCFig9b}
\end{figure}

\begin{figure}[tbp]
\centering
\includegraphics[width = 0.8\textwidth]{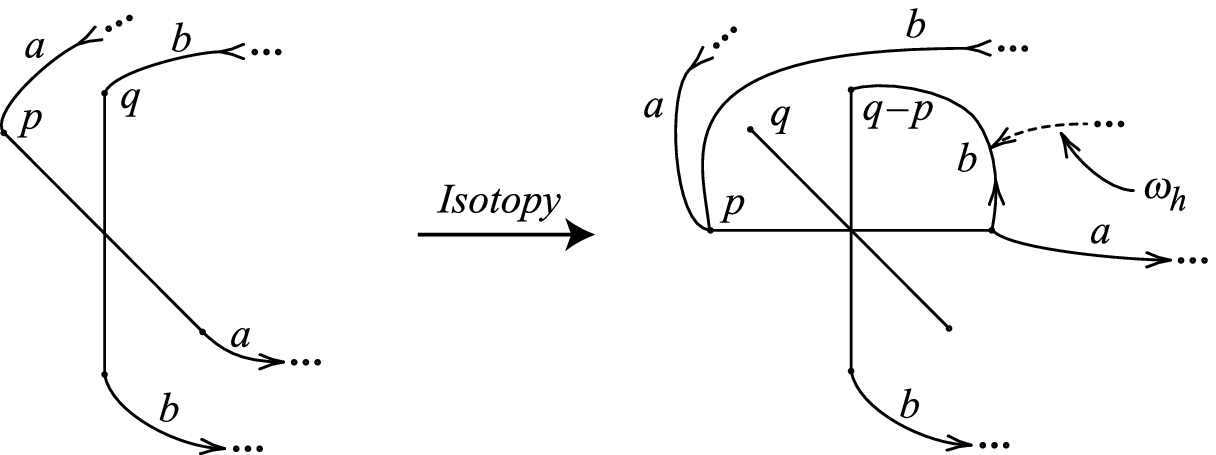}
\caption{Locating the attachment point of a horizontal wave $\omega_h$, meeting an oriented curve from the right, on a handle with two nonempty bands when $q > p > 0$, $a, b > 0$, and $\gcd(p,q) = 1$. (Note the difference between this figure and Figure \ref{DPCFig9b}, which shows the situation when $p > q > 0$.)}
\label{DPCFig9c}
\end{figure}

\begin{figure}[tbp]
\centering
\includegraphics[width = 0.7\textwidth]{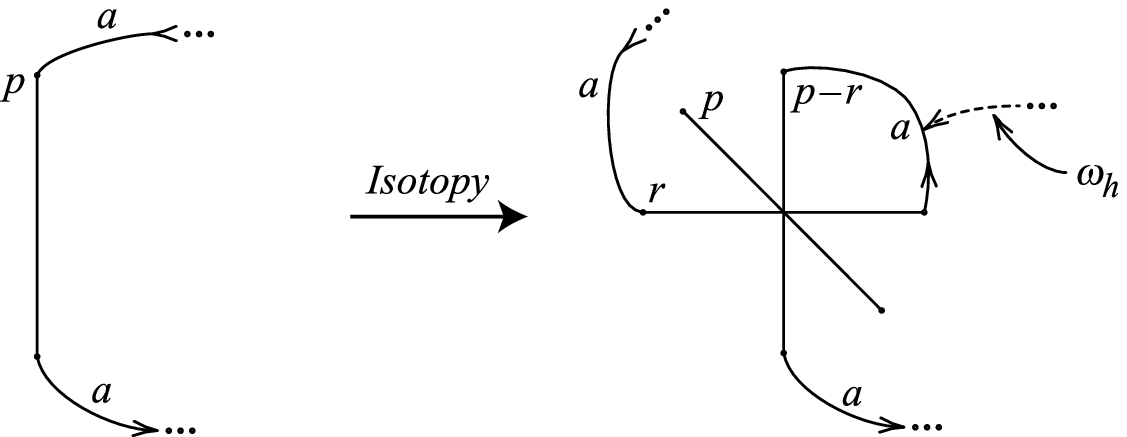}
\caption{Locating the attachment point of a horizontal wave $\omega_h$, meeting an oriented curve from the right, on a handle with only one nonempty band. Here $a > 0$, $p > r > 0$ and $\gcd(p,r) = 1$.}
\label{DPCFig9d}
\end{figure}

\begin{figure}[tbp]
\centering
\includegraphics[width = 0.55\textwidth]{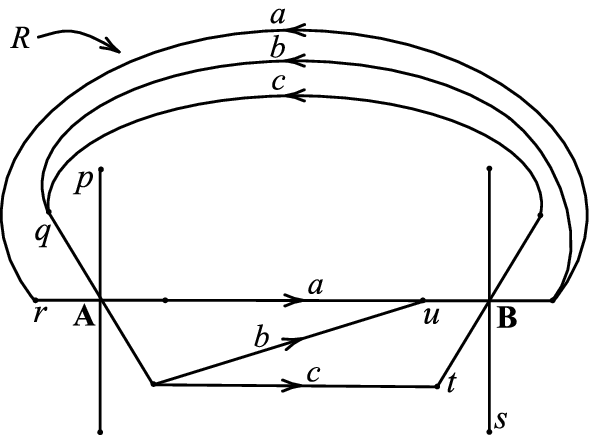}
\caption{An R-R diagram of a nonseparating curve $R$.}
\label{DPCFig9g}
\end{figure}

In order to explain more explicitly the difference between the $p > q$ case in Figure \ref{DPCFig9b} and the $p < q$ case in Figure  \ref{DPCFig9c}, we consider a practical  example.
Suppose $R$ has an R-R diagram $\mathcal{D}_R$ of the form in Figure~\ref{DPCFig9g}, where $a, b, c, q, r, u, t>0$, gcd$(q,r)=$gcd$(u,t)=1$, and $p=q-r, s=t-u$. One is interested in determining the range of values of the parameters $q$, $r$, $u$, and $t$ of $\mathcal{D}_R$ such that the manifold $H[R]$, obtained by adding a 2-handle to a genus two handlebody $H$ along $R$, embeds in $S^3$. The Heegaard diagram $\mathbb{D}_R$ underlying $\mathcal{D}_R$ will be positive, connected and have no cut-vertex if $q + r > 2$ and $u + t > 2$.
If $H[R]$ embeds in $S^3$, then surgery on $R$ along a horizontal wave $\omega_h$ cuts $R$ into two disjoint simple closed curves $M_1$ and $M_2$, such that each of $M_1$, $M_2$ represents the meridian of $H[R]$. However, since $A$ appears in $R$ with only two distinct exponents, there are two subcases to consider: $q>r$ and $q<r$, and because $B$ also appears in $R$ with only two distinct exponents, there are two more subcases to consider: $u>t$ and $u<t$, for a total of four subcases to locate the horizontal wave $\omega_h$. Figure~\ref{DPCFig9h} illustrates how to locate the horizontal wave $\omega_h$ in these four subcases:
a) $q>r$ and $u>t$, b) $q>r$ and $u<t$, c) $q<r$ and $u>t$, d) $q<r$ and $u<t$.

For example, in Figure~\ref{DPCFig9h}a, since $q>r$ and $u>t$, $q$ and $u$ are the maximal labels in the $A$- and $B$-handles respectively. Since $\omega_h$ has an endpoint on a connection in one handle which borders the band of connections with maximal label, we isotope the outermost arc in the set of the $b$ parallel arcs entering the band of the $q$-connections in the $A$-handle and also isotope the outermost edge of the band of width $b$ entering the $u$-connection in the $B$-handle as shown in Figure~\ref{DPCFig9h}a.

\begin{figure}[tbp]
\centering
\includegraphics[width = 1\textwidth]{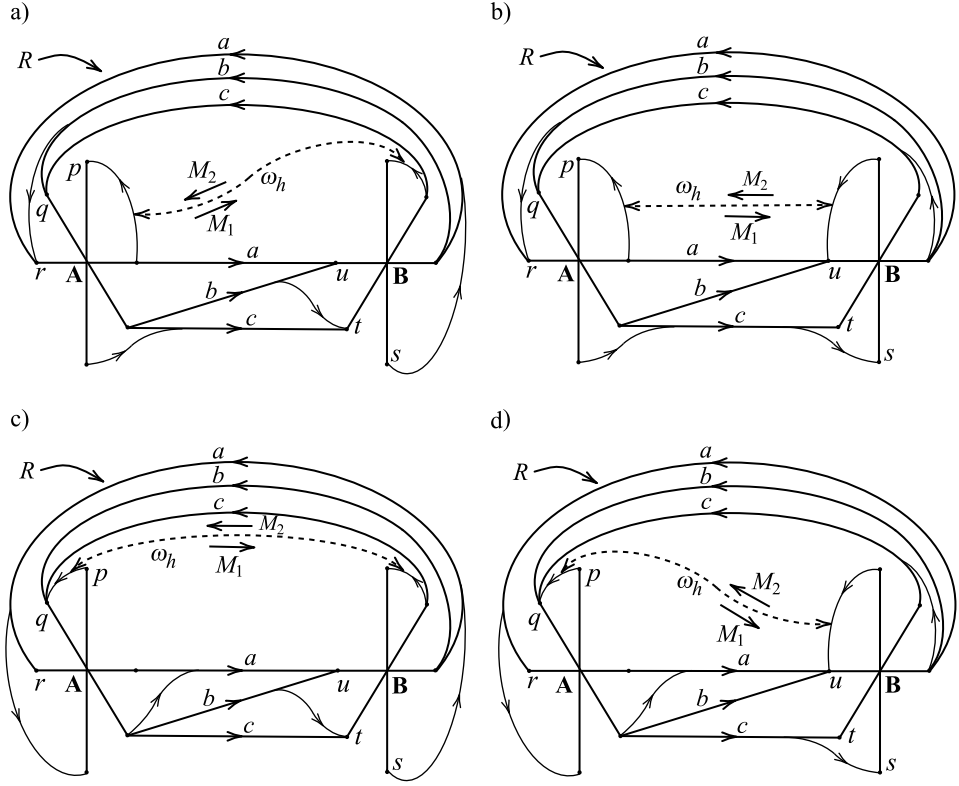}
\caption{The four subcases to locate the horizontal wave $\omega_h$: a) $q>r$ and $u>t$, b) $q>r$ and $u<t$, c) $q<r$ and $u>t$, d) $q<r$ and $u<t$.}
\label{DPCFig9h}
\end{figure}

If $a=b=c=1$ in $\mathcal{D}_R$ of Figure~\ref{DPCFig9g}, then $R = A^rB^uA^qB^tA^qB^u$ and horizontal waves leads to the following four  $(M_1, M_2)$ pairs.
\begin{enumerate}
\item If $q > r$ and $u > t$,  $(M_1, M_2)$ = $(A^rB^{-s}A^rB^{u}, \,B^tA^pB^tA^{q})$.
\item If $q > r$ and $u < t$,  $(M_1, M_2)$ = $(A^rB^{u}A^qB^{u}A^rB^u, \,B^sA^p)$.
\item If $q < r$ and $u > t$,  $(M_1, M_2)$ = $(A^{-p}B^{-s}, \,B^tA^qB^uA^qB^tA^q)$.
\item If $q < r$ and $u < t$,  $(M_1, M_2)$ = $(A^{-p}B^uA^qB^u, \,B^{s}A^qB^uA^q)$.
\end{enumerate}

\section{Culling lists of knot exterior candidates}\label{Culling lists of knot exterior candidates}

As usual, we suppose $R$ is a nonseparating simple close curve on the boundary of a genus two handlebody $H$ with a complete set of cutting disks  $\{D_A, D_B\}$ such
that the Heegaard diagram of $R$ with respect to $\{\partial D_A, \partial D_B\}$ is connected and has no cut-vertex. If $H[R]$ embeds in $S^3$, then Theorem~\ref{waves provide meridians} indicates that there are two representatives $M_1$ and $M_2$ of the meridian of $H[R]$, which are obtained from $R$ by surgery along a distinguished wave based at $R$. One of the results of \cite{B20} shows that $M_1$ and $M_2$ are ``shortest" meridian representatives of $H[R]$. In other words, if a simple closed curve $m$ on $\partial H$ disjoint from $R$ represents the meridian
of $H[R]$, then the complexity of the Heegaard diagram of $m$ with respect to $\{\partial D_A, \partial D_B\}$ is greater than or equal to the complexity of the Heegaard diagram
of either $M_1$ and $M_2$ with respect to $\{\partial D_A, \partial D_B\}$.

Now we suppose that $2|M_1|\leq |R|$, which means that $M_1$ is the shortest meridian
representative of $H[R]$. Then some restriction can be put as an obstruction of being the shortest meridian of $H[R]$ under some circumstances. This section is devoted to make
such restriction.

We start with basic properties of nonseparating simple closed curves in a closed orientable surface.

\begin{lem}
\label{nonparallel implies nonseparating}
Let $\alpha$ and $\beta$ be two disjoint nonseparating simple closed curves in a closed orientable surface $\Sigma$ of genus two. Then either $\alpha$ and $\beta$ are isotopic in $\Sigma$, or the union of $\alpha$ and $\beta$ does not separate $\Sigma$.
\end{lem}

\begin{proof}
Suppose the union of $\alpha$ and $\beta$ separates $\Sigma$. We will show this implies $\alpha$ and $\beta$ are isotopic in $\Sigma $.

Let $N_\alpha$ be a regular neighborhood of $\alpha$ in $\Sigma $ and let $F = \Sigma - \text{int}(N_\alpha)$. Then $F$ is a twice-punctured torus and $\beta$ is a simple closed curve in $F$. Since we have assumed the union of $\alpha$ and $\beta$ separates $\Sigma$, $\beta $ must separate $F$. Let $N_\beta $ be a regular neighborhood of $\beta $ in $F$. Then $F - \text{int}(N_\beta)$ will have two components $C_1$ and $C_2$, say.  Since $\beta$ does not separate $\Sigma$, each of $C_1$ and $C_2$ must have one copy of $\alpha$ and one copy of $\beta$ as its boundary components. Then, turning to Euler characteristics, we have: $\chi(F) = -2$, while each of $\chi(C_1)$, $\chi(C_2)$ is even and $\leq 0$. But $\chi(F) = \chi(C_1) + \chi(C_2)$. Thus either $\chi(C_1) = 0$ or $\chi(C_2) = 0$. So one of $C_1$, $C_2$ is an annulus, and $\alpha$ and $\beta$ are isotopic in $\Sigma$.
\end{proof}

\begin{lem}[\textbf{Unique Disjoint Cutting Disk}]\label{unique disjoint cutting disk}
Let $H$ be a genus two handlebody and let $\alpha$ be a simple closed curve in $ \partial H$ which does not bound a disk in $H$. Let $N_\alpha$ be a regular neighborhood of $\alpha$ in $\partial H$. Then there is at most one nonseparating simple closed curve in $ \partial H - \text{int}(N_\alpha)$ which bounds a disk in $H$.
\end{lem}

\begin{proof}
First, suppose that $\alpha $ separates $ \partial H $. Then each component of $ \partial H -\text{int}(N_\alpha)$ is a punctured torus and if a simple closed curve in either of these components bounds a disk in $H$, then so does $ \alpha $. So, in this case, there are no nonseparating simple closed curves in $ \partial H - \text{int}(N_\alpha)$ which bound disks in $H$, and the result holds.

Next suppose $ \alpha $ is nonseparating in $ \partial H $, and there exist two nonseparating nonisotopic simple closed curves $\beta $ and $\gamma $, say, in $ \partial H - \text{int}(N_\alpha)$, each of which bounds a disk in $H$. We will show this implies $\alpha $ bounds a disk in $H$, contrary to hypothesis.

Note Lemma~\ref{nonparallel implies nonseparating} implies that if $\beta$ and $\gamma$ are disjoint in $\partial H $, and not isotopic in $\partial H $, then $\beta $ and $\gamma $ bound a set of cutting disks for $H $. But this implies $\alpha $ also bounds a disk in $H $, contrary to hypothesis.

So $\beta $ and $\gamma $ must intersect essentially. Let $D_\beta$ and $D_\gamma$ be disks properly embedded in $H$ which are bounded by $\beta $ and $\gamma $ respectively.
We may assume that $D_\beta$ and $D_\gamma$ intersect minimally and only in arcs. Let $F_\gamma$ be an outermost subdisk of $D_\gamma $ cut off of $D_\gamma $ by the arcs in $D_\beta \cap D_\gamma $. Since, by Lemma~\ref{nonparallel implies nonseparating}, $\alpha$ and $\beta $ do not separate $\partial H$, we can form a ``Heegaard'' diagram $\mathbb{D}_\gamma$ of $\gamma $ by cutting $\partial H $ open along $\alpha $ and $\beta $. By Lemma~\ref{3 types of genus two diagrams}, $\mathbb{D}_\gamma$ must have the form shown in Figure~\ref{DPCFig1s}. Then an examination of $\mathbb{D}_\gamma $ shows that $F_\gamma $ can be used to perform surgery on $D_\beta $ so as to obtain a simple closed curve $\delta $ in $\partial H $ such that $\delta $ bounds a disk in $H$ and $\delta $ is isotopic to $\alpha $. This contradicts the assumption that $\alpha $ does not bound a disk in $H$ and completes the proof.
\end{proof}

\begin{figure}[tbp]
\centering
\includegraphics[width = 0.35\textwidth]{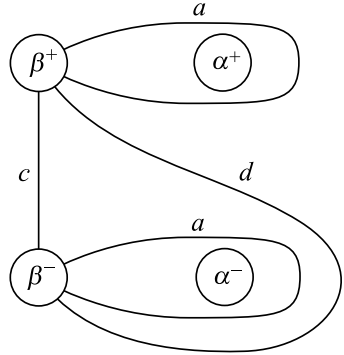}
\caption{The Heegaard diagram $\mathbb{D}_\gamma$. Weights $a > 0$, $c,d \geq 0$}
\label{DPCFig1s}
\end{figure}

As examples of Lemma~\ref{unique disjoint cutting disk}, let $\alpha$ be a primitive curve in a genus two handlebody $H$, i.e., $H[\alpha]$ is a solid torus.
Then there is a cutting disk which intersects $\alpha$ transversely at a point and thus there is another cutting disk $D$ which is disjoint from $\alpha$. By Lemma~\ref{unique disjoint cutting disk}, this cutting disk $D$ is unique. On the other hand, if $\alpha$ is a Seifert curve in $H$, i.e., $H[\alpha]$ is a Seifert-fibered space and not a solid torus,
then $\alpha$ must intersect every cutting disk of $H$ since $H[\alpha]$ is $\partial$-irreducible. Thus there is no cutting disk in $H$ which is disjoint from $\alpha$.

Now we describe the main result in this section which will be used in Section~\ref{R-R diagrams of a tunnel-number-one knot exteriors with only one band of connections in a handle}. Theorem~\ref{waves provide meridians}
guarantees that a meridian $M$ of $H[R]$ can be obtained from $R$ by surgery along a distinguished wave $\omega$, provided that the Heegaard diagram $\mathbb{D}_R$ is connected and has no cut-vertex.
Note that $H[M]$ is also a knot exterior of $S^3$ and thus Theorem~\ref{waves provide meridians} can apply to $M$ as well.
The following lemma, which is called the \textit{Culling Lemma}, provides some restriction on the shortest meridian of $H[R]$.

\begin{lem}[\textbf{The Culling Lemma}]\label{culling lemma}

Let $H$ be a genus two handlebody with a complete set of cutting disks $\{D_A, D_B\}$ and let $\alpha$ and $R$ be two disjoint nonseparating simple closed curves on $\partial H$. Suppose that $H[R]$ embeds in $S^3$ and the Heegaard diagram of $R$ with respect to $\{D_A, D_B\}$ is connected and has no cut-vertex, and has a distinguished wave $\omega$ such that $\alpha$ and $\omega$ intersect essentially in a single point. Let $M_1$ be the representative of the meridian of $H[R]$ with $2|M_1|\leq |R|$, obtained from $R$ by surgery along $w$.

Suppose the Heegaard diagram of $M_1$ with respect to $\{D_A, D_B\}$ is connected and has no cut-vertex, and has a distinguished wave $\omega_1$. Then $\omega_1$ must have more than one essential intersection with $\alpha$. Furthermore if $\omega_1$ has exact two essential intersections with $\alpha$, then they have same signs.
\end{lem}

\begin{proof}
Since $H[R]$ embeds in $S^3$ and $M_1$ is the representative of the meridian of $H[R]$, $R$ and $M_1$ bound a set of cutting disks of a genus
two handlebody $H'$ in $S^3$. Let $M_2$ be the other representative of the meridian of $H[R]$, obtained by surgery on $R$ along $\omega$. Since $2|M_1|\leq |R|$ and $|M_1|+|M_2|\leq |R|$, $|M_1|\leq |M_2|$, and, since $S^3$ does not contain any nonseparating $S^2$s, $|M_i| > 0$ for $i = 1,2$.

First we prove that $\omega_1$ has more than one essential intersection with $\alpha$.
If $\omega_1$ is disjoint from $\alpha$, one of the two simple closed curves obtained from $M_1$ by surgery on $M_1$ along $\omega_1$ is disjoint from $\alpha$. Let $R'$ be that curve. Then $R'$ is a nonseparating simple closed curve in $\Sigma$, $R'$ bounds a disk in $H'$, and $|R'| < |M_1| < |R|$. But this is impossible since by Lemma~\ref{unique disjoint cutting disk} $R$ is the unique (up to isotopy) simple closed curve in $\Sigma$ which is disjoint from $\alpha$ and bounds a disk in $H'$. So $\omega_1$ must intersect $\alpha$ at least once.

Now we show that $\omega_1 \cap \alpha$ cannot consist of just one point of essential intersection. Suppose, to the contrary, that $\omega_1 \cap \alpha$ consists of a single point of essential intersection. Then surgery on $M_1$ along $\omega_1$ yields two nonseparating simple closed curves $M_{1,1}$ and $M_{1,2}$, each of which bounds a disk in $H'$. One of $M_{1,1}$ and $M_{1,2}$ intersects $\alpha$ once, while the other curve intersects $\alpha$ twice. Let $M_{1,1}$ be the curve which intersects $\alpha$ once. Then $M_1$ and $M_{1,1}$ bound a complete set of cutting disks of $H'$ and the Heegaard diagram of $\alpha$ with respect to $M_1$ and $M_{1,1}$ has the form shown in Figure \ref{CullFig1} up to the orientations of $M_1$ and $M_{1,1}$.

\begin{figure}[tbp]
\centering
\includegraphics[width = 0.4\textwidth]{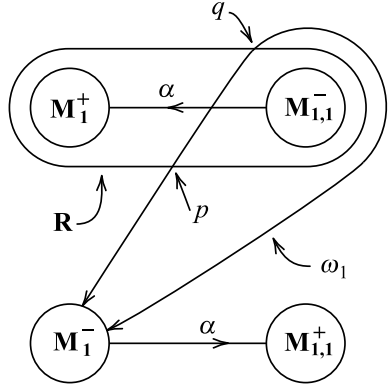}
\caption{The handlebody $H'$ cut open along the set of cutting disks bounded by $M_1$ and $M_{1,1}$. Here $|M_1 \cap \alpha| = 1$, $\omega_1$ is a wave based at $M_1$ intersecting $\alpha$ only once, so $|M_{1,1} \cap \alpha|$ = $1$, and $R$ is the bandsum of $M_1$ and $M_{1,1}$ along a subarc of $\alpha$. (Note the subarc of $\omega_1$ with endpoints at $p$ and $q$ is a wave $\omega'$ based at $R$.)}
\label{CullFig1}
\end{figure}

Since $\alpha$ intersects each of $M_1$ and $M_{1,1}$ only once, uniqueness of $R$ implies $R$ is the curve in Figure \ref{CullFig1} obtained by bandsumming $M_1$ and $M_{1,1}$ together along a subarc of $\alpha$. Now let $\omega'$ be the subarc of $\omega_1$ as shown in Figure \ref{CullFig1} which has its endpoints at the points $p$ and $q$ of $\omega_1 \cap R$. Then $\omega'$ is a wave based at $R$. We claim $\omega'$ is isotopic to $\omega$.

To see this, if the Heegaard diagram of $R$ is nonpositive with respect to $\{D_A, D_B\}$, then since by Lemma 3.6 in \cite{B20} $R$ has a unique wave, $\omega'$ is isotopic to $\omega$. But this implies $M_2$ is isotopic to $M_{1,1}$, which is impossible since $|M_{1,1}| < |M_1| \leq |M_2|$. So $R$ must be positive with respect to $\{D_A, D_B\}$.

Next, if the Heegaard diagram of $R$ with respect to $\{D_A, D_B\}$ is positive, then $R$ has both horizontal and vertical waves, where a vertical wave is an arc one of whose endpoints lies on an edge of the Heegaard diagram $\mathbb{D}_R$ of $R$ connecting $A^+$ to one member of $\{B^+, B^-\}$, while the other endpoint lies on an edge of $\mathbb{D}_R$ connecting $A^-$ to the other member of $\{B^+, B^-\}$ and a horizontal wave is distinguished to be used
to get a meridian of $H[R]$ by Theorem~\ref{waves provide meridians}. Note that a vertical wave and a horizontal wave in the positive Heegaard diagram of $R$ intersect each other once transversely. So $\omega'$ is either horizontal or vertical. Now $M_1$ is obtained from $R$ by surgery along a horizontal wave, namely $\omega$, and, as Figure \ref{CullFig1} shows, surgery on $R$ along $\omega'$ yields curves disjoint from $M_1$. This forces $\omega'$ to be horizontal, since surgery on $R$ along a vertical wave would yield a simple closed curve having a single point of essential intersection with $M_1$.

Thus, again $\omega'$ is isotopic to $\omega$. So, as before, $M_2$ must be isotopic to $M_{1,1}$. Since this is impossible, we conclude $\omega_1$ must have more than one essential intersection with $\alpha$.

\begin{figure}[tbp]
\centering
\includegraphics[width = 0.9\textwidth]{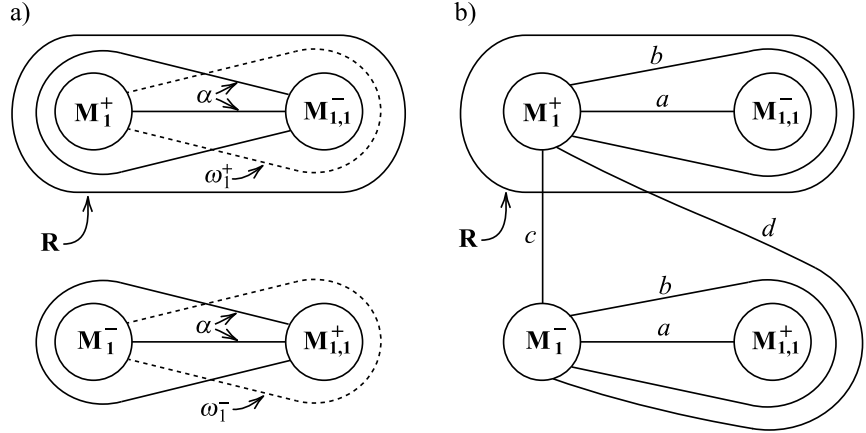}
\caption{The Heegaard diagram of $\alpha, \omega_1$, and $R$ with respect to $\{M_1, M_{1,1}\}$ in Figure~\ref{CullFig2}a and the Heegaard diagram
of $\partial D_A\cup \partial D_B$ with respect to $\{M_1, M_{1,1}\}$ in Figure~\ref{CullFig2}b.}
\label{CullFig2}
\end{figure}

Now we prove the second part of the lemma, i.e, that if $\omega_1$ has exact two essential intersections with $\alpha$, then they have same signs.
Suppose $\omega_1$ has exact two essential intersections with $\alpha$ with opposite intersection signs. As before, let $M_{1,1}$ and $M_{1,2}$ be
two meridian representatives of $H[M_1]$ obtained by surgery on $M_1$ along $\omega_1$ such that $0<|M_{1,1}|<|M_1|$ and $0<|M_{1,2}|<|M_1|$. Since $M_1$
has one essential intersection with $\alpha$ and $\omega_1$ has two essential intersections with $\alpha$, one of $M_{1,1}$ and $M_{1,2}$
intersects $\alpha$ twice and the other member intersects $\alpha$ three times. Suppose $M_{1,1}$ intersects $\alpha$ three times.
Then it is not hard to see that the Heegaard diagram of $\alpha$ and $\omega_1$ with respect to $M_1$ and $M_{1,1}$ has the form of Figure~\ref{CullFig2}a, where $\omega_1$
is either $w^+_1$ or $w^-_1$.
Note that since $R$ is the unique nonseparating simple closed curve disjoint from $\alpha$ which bounds a disk in $H'$, $R$ appears as shown in Figure~\ref{CullFig2}a.

Now consider the diagram obtained by adding $\partial D_A$ and $\partial D_B$ to Figure~\ref{CullFig2}a. Noting that $\omega_1$ is disjoint from $\partial D_A$ and $\partial D_B$,
while $0<|M_{1,1}|<|M_1|$, we see that the Heegaard diagram of $\partial D_A\cup \partial D_B$ has the form of Figure~\ref{CullFig2}b. Since $0<|M_{1,1}|<|M_1|$, $a>0$ in Figure~\ref{CullFig2}b. It follows that $|R|=c+d$ and $|M_1|=a+2b+c+d$. This is impossible since $|M_1|<|R|$. We conclude that if $\omega_1$ has exact two essential intersections with $\alpha$, then they have same signs.
\end{proof}

The Culling Lemma is very useful in the procedure of the classification of hyperbolic P/SF knots in $S^3$. From the first and the second steps for the classification, we classify all possible R-R diagrams of disjoint simple closed curve $\alpha$ and $R$ lying on the boundary of a genus two handlebody such that $\alpha$ is Seifert in $H$ and $H[R]$ is homeomorphic to the exterior of a hyperbolic knot which is represented by $\alpha$. Next step is to find a meridian representative $M$ of $H[R]$, which can be obtained from $R$ by surgery along a distinguished wave. Then the Culling Lemma applies to cull out candidates of $M$. We give an explicit example of the application of the Culling Lemma in the classification of hyperbolic P/SF knots in $S^3$.

\begin{figure}[tbp]
\centering
\includegraphics[width =0.55\textwidth]{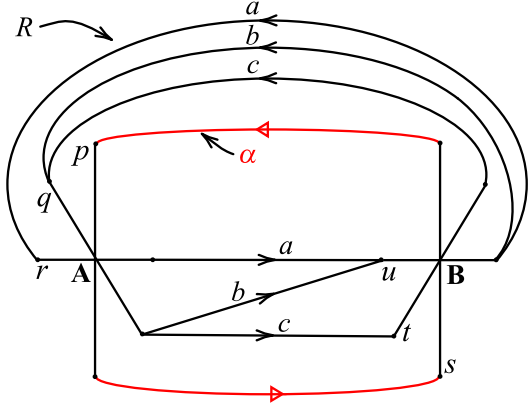}
\caption{R-R diagram of two disjoint nonseparating simple closed curves $\alpha$ and $R$, where $q>r>0$ and $u>t>0$.}
\label{CullFig3}
\end{figure}

\begin{figure}[tbp]
\centering
\includegraphics[width = 1\textwidth]{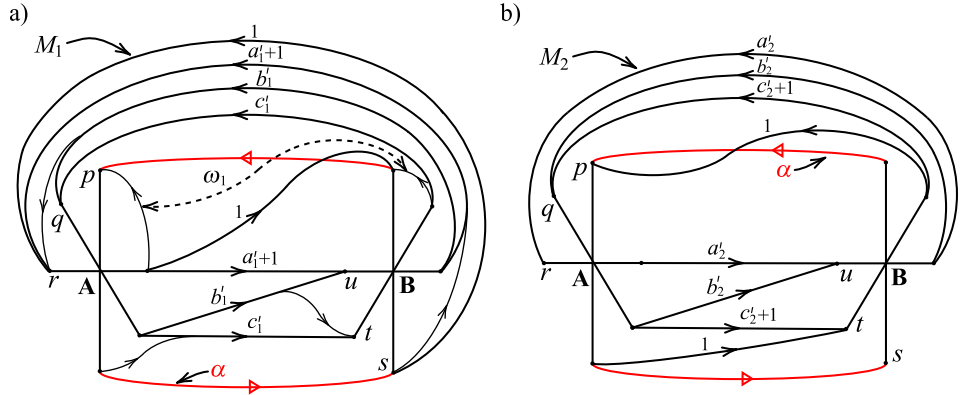}
\caption{Two meridian representatives $M_1$ and $M_2$ of $H[R]$ obtained from $R$ by surgery along the horizontal wave $\omega_h$ and a horizontal wave $\omega_1$ based at $M_1$ which intersects $\alpha$ once in a).}
\label{CullFig4}
\end{figure}

Suppose that two disjoint simple closed curves $\alpha$ and $R$ have an R-R diagram of the form shown in Figure~\ref{CullFig3}, which
is the R-R diagram of $R$ in Figure~\ref{DPCFig9g} with the curve $\alpha$ added. Note that the curve $\alpha$ is a Seifert curve such that $H[\alpha]$ is a Seifert-fibered space over the disk with two exceptional fibers of indexes $|p|$ and $|s|$. Suppose $q>r>0$ and $u>t>0$. Then by surgery on $R$ along the horizontal
wave given in Figure~\ref{DPCFig9h}a, the two meridian representatives $M_1$ and $M_2$ have R-R diagrams as illustrated in Figure~\ref{CullFig4}. Suppose $|M_1|\leq |M_2|$, i.e., $2|M_1|\leq |R|$.

We will show by applying for the Culling Lemma that $b'_1=c'_1=0$ in the R-R diagram of $M_1$. Suppose $b'>0$ in Figure~\ref{CullFig4}a. Then $M_1$ has connections with the maximal labels $q$ and $u$ in the $A$- and $B$-handles respectively and thus the Heegaard diagram of $M_1$ underlying the R-R diagram is connected and has no cut-vertex. Therefore there is a horizontal wave $\omega_1$ based at $M_1$ whose endpoints lie at the $q$-connection in the $A$-handle and the $u$-connection in the $B$-handle. By isotoping the both connections as in Figure~\ref{CullFig4}a, the horizontal wave $\omega_1$ based at $M_1$ appears as in Figure~\ref{CullFig4}a and intersects $\alpha$ transversely at a point. This is a contradiction to the Culling Lemma. So $b'_1$ must be $0$ and also $c'_1=0$ in the R-R diagram of $M_1$, otherwise $M_1$ can't be a simple closed curve.

The similar argument can apply to the other representative $M_2$ to show that $b'_2=c'_2=0$ when $|M_2|\leq |M_1|$.

\section{R-R diagrams of tunnel-number-one knot exteriors disjoint from proper power curves}
\label{R-R diagrams of a tunnel-number-one knot exteriors with only one band of connections in a handle}

In this section, we prove the main result of this paper, which is described in Theorem~\ref{main theorem 1}. In order to prove it, we consider a simple closed curve $R$ in the boundary of a genus two handlebody $H$ which has an R-R diagram with only one band of connections in one handle, say, $B$-handle and at most three bands of connections in the $A$-handle as shown in Figure~\ref{DPCFig21as} with $a,b,c\geq 0$. Note that if $s>1$ in Figure~\ref{DPCFig21as}, there exists a proper power curve $\beta$ as shown there  disjoint from $R$ with $[\beta]=B^s$ in $\pi_1(H)$. In Subsections~\ref{R-R diagrams of $R$ with one connection on one handle and three connections on the other handle} and ~\ref{S: R-R diagrams of torus and cable knots}, we handle R-R diagrams of $R$ of the form in Figure~\ref{DPCFig21as} with $a,b,c>0$ and with at least one of $a,b,c$ being $0$ but not all respectively. Throughout this section, for two disjoint nonseparating simple closed curves $c_1$ and $c_2$ in the boundary of a genus two handlebody $H$, $H[c_1, c_2]$ denotes
the 3-manifold obtained by adding a 2-handle to $H$ along $c_1$ and $c_2$ each and by filling in it sphere boundary with a 3-ball.

\subsection{R-R diagrams of $R$ with one connection in one handle and three connections in the other handle}
\label{R-R diagrams of $R$ with one connection on one handle and three connections on the other handle}\hfill
\smallskip

We consider an R-R diagram of $R$ which has only one band of connections in one handle and has three bands of connections in the other handle, i.e, an R-R diagram of $R$ of the form shown in Figure~\ref{DPCFig21as} with $a,b,c>0$. Then the necessary condition for $H[R]$ to embed in $S^3$ is given in the following Theorem.

\begin{figure}[tbp]
\centering
\includegraphics[width = 0.6\textwidth]{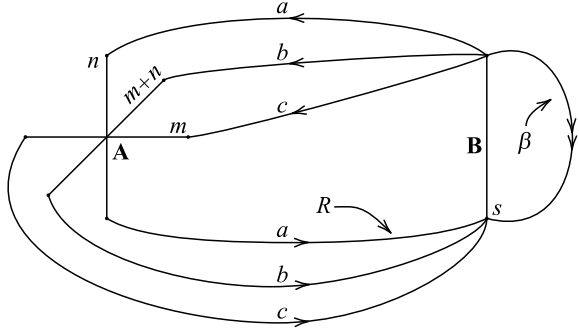}
\caption{An R-R diagram of $R$ where $R$ has only one band of connections in one handle and has three bands of connections in the other handle.}
\label{DPCFig21as}
\end{figure}

\begin{thm}
\label{at least one of a, b, c is zero}
Suppose $R$ is a simple closed curve with an R-R diagram of the form shown in Figure~\emph{\ref{DPCFig21as}} with $m,n\neq 0, (m,n)\neq(\pm1,\pm1)$ with the double signs in same order, $\gcd(m,n) = 1$, and $|s| > 1$. If $a$, $b$, and $c$ are all positive in Figure~\emph{\ref{DPCFig21as}}, then $H[R]$ cannot embed in $S^3$.
\end{thm}

\begin{proof}
Suppose, to the contrary, that there is a simple closed curve $R$ which has an R-R diagram with the form shown in Figure~\ref{DPCFig21as}, with each of $a$, $b$, and $c$ positive, and $H[R]$ embeds in $S^3$. We may assume that $s>1$ because the R-R diagram of $R$ with $s<-1$ is equivalent to that with $s>1$. We divide the argument into two cases: $mn>0$ and $mn<0$.

First, suppose $mn<0$. Then the Heegaard diagram $\mathbb{D}_R$ of $R$ underlying the R-R diagram in Figure~\ref{DPCFig21as} is nonpositive. Also we see that $\mathbb{D}_R$ is connected and has no cut-vertex. Therefore by Theorem~\ref{waves provide meridians}, one meridian representative of $H[R]$ can be obtained by surgery on $R$ along a vertical wave $\omega_v$. It is easy to see that one meridian representative is a simple closed curve $\beta$ as shown in Figure~\ref{DPCFig21as}. This is a contradiction since $\beta=B^s$ in $\pi_1(H)$ with $s>1$ and $H[\beta]$ embeds in $S^3$ implying that $H_1(H[\beta])$ is torsion-free.

Now, we suppose $mn>0$. The Heegaard diagram $\mathbb{D}_R$ of $R$ is positive, connected, and has no cut-vertex. By using the equivalence of R-R diagrams, we may assume that $m,n>0$. Since $H[R]$ embeds in $S^3$, Theorem~\ref{waves provide meridians} implies that two representatives of the meridian of $H[R]$ can be obtained by surgery on $R$ along a horizontal wave $\omega_h$. By the results of Section~\ref{Locating waves in genus two R-R diagrams} a horizontal wave $\omega_h$ based at $R$ appears in the R-R diagram as shown in Figure~\ref{DPCFig22as}, where $0<u<s$.

\begin{figure}[tbp]
\centering
\includegraphics[width = 0.65\textwidth]{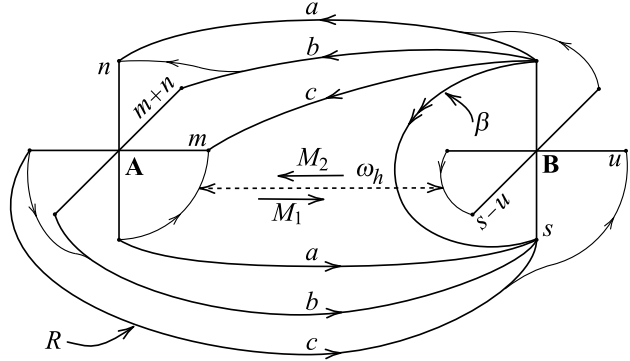}
\caption{A horizontal wave $\omega_h$. }
\label{DPCFig22as}
\end{figure}

\begin{figure}[tbp]
\centering
\includegraphics[width = 1\textwidth]{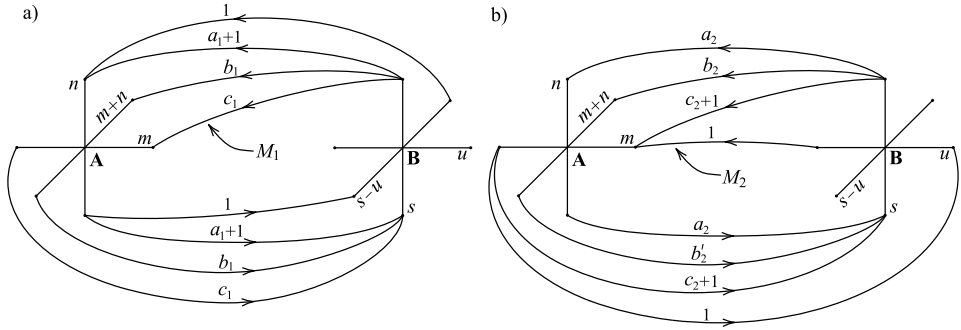}
\caption{Two representations $M_1$ and $M_2$ of the meridian of $H[R]$. }
\label{DPCFig22ds}
\end{figure}

Let $M_1$ and $M_2$ be the two representatives of the meridian of $H[R]$ obtained by surgery on $R$ along $\omega_h$.
Since $a, b, c>0$, their R-R diagrams have the form shown in Figure~\ref{DPCFig22ds}. Then examination of Figure~\ref{DPCFig22ds} shows that $B^s$ appears in both $M_1$ and $M_2$.
Also note that since $m,n,s,u, s-u>0$, the corresponding Heegaard diagrams
of $M_1$ and $M_2$ are both positive.
Now Lemmas~\ref{contradiction if neither m1 or m2 has a cut-vertex}, \ref{case m1 has a cut-vertex} and \ref{case m2 has a cut-vertex} complete the proof.
\end{proof}

\begin{lem}
\label{contradiction if neither m1 or m2 has a cut-vertex}
If $|M_1| \leq |M_2|$ and $M_1$ does not have a cut-vertex, or if $|M_2| \leq |M_1|$ and $M_2$ does not have a cut-vertex, then $H[R]$ does not embed in $S^3$.
\end{lem}

\begin{proof}
Suppose $|M_1| \leq |M_2|$ and $M_1$ does not have a cut-vertex. Since $M_1$ is positive, there is a horizontal wave $\omega'_{h}$ based at $M_1$. However, it is easy to see from Figure~\ref{DPCFig22ds}a that because $B^s$ appears in $M_1$, as the horizontal wave $\omega_h$ in Figure~\ref{DPCFig22as} $\omega'_{h}$ intersects $\beta$ in a single point. This is a contradiction to
the Culling Lemma of Lemma~\ref{culling lemma}.

A similar argument works to show that if $|M_2| \leq |M_1|$, and $H[R]$ embeds in $S^3$, then $M_2$ must have a cut-vertex.
\end{proof}

\begin{lem}
\label{case m1 has a cut-vertex}
If $M_1$ has a cut-vertex, then $H[R]$ does not embed in $S^3$.
\end{lem}

\begin{proof}
$M_1$ has the R-R diagram of the form in Figure~\ref{DPCFig22ds}a, where $A^n$ and $B^s$ appear in $M_1$. Because $s > 1$ and $(m,n) \neq (1,1)$, $M_1$ can have a cut-vertex only if $A$ appears in $M_1$ only with exponent one and thus $b_1=c_1=0$ and $n=1$ in Figure~\ref{DPCFig22ds}a. It follows that
\begin{equation}
\label{form of m1 with a cut-vertex}
M_1 = B^{s-u}A(B^sA)^{a_1+1}.
\end{equation}

\begin{claim}
\label{claim 11}
If $M_1$ has a cut-vertex, and $ u > 1$ in Figure~\emph{\ref{DPCFig22ds}a}, then $H[R]$ does not embed in $S^3$.
\end{claim}

\begin{figure}[tbp]
\centering
\includegraphics[width = 0.65\textwidth]{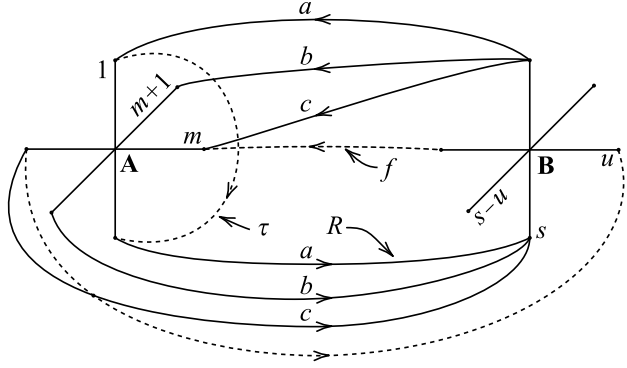}
\caption{The original R-R diagram of $R$ in Figure~\ref{DPCFig21as} with $n$ set equal to one, and with two additional simple closed curves $f$ and $\tau$ added to the diagram.}
\label{DPCFig22bs}
\end{figure}

\begin{proof} [Proof of Claim~\emph{\ref{claim 11}}]
We will show that if $M_1$ has a cut-vertex, and $ u > 1$ in Figure~\ref{DPCFig22ds}a, then $H[M_1]$ is Seifert-fibered over $D^2$ with exceptional fibers of indexes $u$ and $a_1+2$, and then $H[R,M_1]$ is Seifert-fibered over $S^2$ with exceptional fibers of indexes $u$, $a_1+2$ and $m(b+c) - c$. Since $u > 1$ and $a_1+2 > 0$, it will follow that $H[R,M_1]$ is not $S^3$ once we show that $m(b+c) - c > 1$.

To see that $H[M_1]$ is Seifert-fibered and $m(b+c) - c > 1$, consider Figure~\ref{DPCFig22bs}. This figure shows the original R-R diagram of $R$ in Figure~\ref{DPCFig21as} with $n$ set equal to one, and with two additional simple closed curves $f$ and $\tau$ added to the diagram. Notice that if $M_1$ has a cut-vertex, then, since $M_1$ has the form shown in (\ref{form of m1 with a cut-vertex}), i.e., the form shown in ~\ref{DPCFig22ds}a with $b_1=c_1=0$ , $M_1$ is disjoint from both $f$ and $\tau$.

Now let $\tau(f)$ denote the simple closed curve in Figure~\ref{DPCFig22bs} obtained from $f$ by twisting $f$ to the left $m$ times about the curve $\tau$. Then $\tau(f)$ and $M_1$ are disjoint, and $\tau(f)$ represents the proper power $B^u$ in $\pi_1(H)$. Note also that the automorphism of $\pi_1(H)$, which takes $A \mapsto B^{-s}A$ and fixes $B$, takes $M_1 \mapsto B^{-u}A^{a_1+2}$. Thus by Lemma 2.2 of \cite{D03} or Theorem 4.2 of \cite{K20}, $M_1$ clearly has the form of a Seifert fiber relator, i.e., $H[M_1]$ is Seifert-fibered with two exceptional fibers, as claimed, and the simple closed curve $\tau(f)$ represents a regular fiber of $H[M_1]$.

Next, another glance at Figure~\ref{DPCFig22bs} shows that $R$ and $\tau(f)$ have $\pm(m(b+c) - c)$ algebraic intersections. But, $m > 1$ and both $b$ and $c$ are positive, so clearly $m(b+c) - c > 1$. This implies $H[R,M_1]$ is Seifert-fibered over $S^2$ with three exceptional fibers, and so it can't be $S^3$. This concludes the proof of the claim.
\end{proof}

\begin{claim} \label{claim 12}
If $M_1$ has a cut-vertex, and $ u = 1$ in Figure~\emph{\ref{DPCFig22as}}, then $H[R]$ does not embed in $S^3$.
\end{claim}

\begin{proof} [Proof of Claim~\emph{\ref{claim 12}}]
If $u = 1$, then, as (\ref{form of m1 with a cut-vertex}) shows, $M_1$ is primitive in $\pi_1(H)$ and $H[M_1]$ is a solid torus. Then $H[R,M_1]$ will be $S^3$ if and only if it is an integral homology sphere. We will show this never happens when all three weights $a$, $b$ and $c$ of the R-R diagram of $R$ in Figure~\ref{DPCFig21as} are positive.

The proof starts with the R-R diagram in Figure~\ref{DPCFig22cs}. This diagram shows a simple closed curve $M_0$ which intersects $\beta$ transversely in a single point. Note that twisting $M_0$ to the right $a_1+1$ times about the curve $\tau$ in Figure~\ref{DPCFig22cs} transforms $M_0$ into the curve $M_1$ in Figure~\ref{DPCFig22ds}a with $b_1=c_1=0$.

\begin{figure}[tbp]
\centering
\includegraphics[width = 0.65\textwidth]{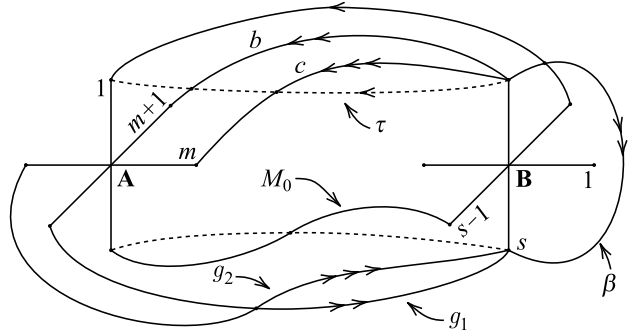}
\caption{The original R-R diagram of $R$ in Figure~\ref{DPCFig21as} with $u$ set equal to one, and with four additional simple closed curves $\beta$, $\tau$, $g_1$ and $g_2$ added to the diagram.}
\label{DPCFig22cs}
\end{figure}

Now let $N$ be a regular neighborhood of $\beta$ and $M_0$ in $\partial H$, and let $N'$ be the punctured torus $\partial H -\text{int}(N)$.
Note that the curves $g_1$ and $g_2$ in Figure~\ref{DPCFig22cs} intersect transversely at a single point and form a spine of $N'$. Then the images $\tau(g_1)$ of $g_1$ and $\tau(g_2)$ of $g_2$ under Dehn twisting $a_1+1$ times about $\tau$ form a spine of a punctured torus in $\partial H$ which contains $R$. So if $R$ in Figure~\ref{DPCFig21as} has positive weights $a$, $b$ and $c$, then $[\vec{R}] = b[\tau(\vec{g}_1)] + c[\tau(\vec{g}_2)]$ in $H_1(\partial H)$.

Turning to $H_1(H)$, and choosing coordinates for $H_1(H)$ so $[\vec{A}] = (1,0)$ and $[\vec{B}] = (0,1)$ and letting $j=a_1+1$, we have
\begin{align}
[\vec{M}_1] &= (1,s-1) +j(1,s) \notag \\
&= (j+1,(j+1)s - 1), \text{\quad and} \notag \\
[\vec{R}] &= (b(m+1), bs) + bj(1,s) + c(m,s) +cj(1,s) \notag \\
&= ((j+m)(b+c) + b, (j+1)(b+c)s). \notag
\end{align}
The weight $a$ of $R$ in Figure~\ref{DPCFig21as} is then given by $a = j(b+c)$, which will be positive if $j > 0$. Then there will exist an $R$, $M_1$ pair with $a,b,c > 0$, such that $H[R,M_1]$ is $S^3$ if and only if it is possible to choose $b,c,j > 0$ and $m > 1$ with $\gcd(b,c) = 1$ so that
\begin{equation}
\label{M,m1 determinant}
(j+1)s[(j+m)(b+c) + b] = (j+1)^2(b+c)s + (j+m)(b+c) + b + \delta,
\end{equation}
with $\delta = \pm1$. (Note that $m > 1$ since $n = 1$ and $(m,n) \neq(1,1)$. The $\gcd(b,c) = 1$ condition arises because $R$ is a simple closed curve; cf Lemma~\ref{gcd(a,b)=1}.)

Subtracting the $(j+1)^2(b+c)s$ term from both sides of (\ref{M,m1 determinant}) yields
\begin{equation}
\label{M,m1 determinant simplified}
(j+1)s[(m-1)(b+c)+b] = (j+m)(b+c) + b + \delta.
\end{equation}

We will show (\ref{M,m1 determinant simplified}) can't be satisfied by showing that the L.H.S. of (\ref{M,m1 determinant simplified}) is always larger than the R.H.S. of (\ref{M,m1 determinant simplified}).

Clearly (\ref{M,m1 determinant simplified}) is the sum of
\begin{gather}
s(j+1)(m-1)(b+c) = (j+m)(b+c), \text{\quad and} \label{inequality 1}\\
bs(j+1) = b + \delta. \label{inequality 2}
\end{gather}

Then, since $m > 1$ and $s > 1$, we have $js(m-1) > j$ and $(s-1)(m-1) \geq 1$. The latter inequality clearly implies $s(m-1) \geq m$. So we have $s(j+1)(m-1) > (j+m)$. Thus the L.H.S. of (\ref{inequality 1}) is always greater than the R.H.S. of (\ref{inequality 1}). Next, since $b > 0$, $j > 0$ and $s > 1$, we have $b[s(j+1) -1] > 1$, which implies the L.H.S. of (\ref{inequality 2}) is always greater than the R.H.S. of (\ref{inequality 2}). It follows that $H[R,M_1]$ is never $S^3$ when $M_1$ has a cut-vertex and $u = 1$. So Claim~\ref{claim 12} has been proved.
\end{proof}
Finally, the proofs of Claims~\ref{claim 11} and \ref{claim 12} combine to prove Lemma~\ref{case m1 has a cut-vertex}.
\end{proof}

\begin{lem}
\label{case m2 has a cut-vertex}
Suppose $M_2$ is the representative of the meridian of $H[R]$ obtained by surgery on $R$ along the horizontal wave $\omega_h$ shown in Figure~\emph{\ref{DPCFig22as}}.
If $M_2$ has a cut-vertex, then $H[R]$ does not embed in $S^3$.
\end{lem}

\begin{proof}
Note that the symmetry of the original diagram of $R$ in Figure~\ref{DPCFig21as} shows that the arguments of Lemma~\ref{case m1 has a cut-vertex}, used in the case that $M_1$ has a cut-vertex, will also work for the case $M_2$ has a cut-vertex, provided minor notational changes are made.
\end{proof}

The following two lemmas deal with the exceptions for the parameters $m$ and $n$ in
Theorem~\ref{at least one of a, b, c is zero}.

\begin{figure}[tbp]
\centering
\includegraphics[width = 1\textwidth]{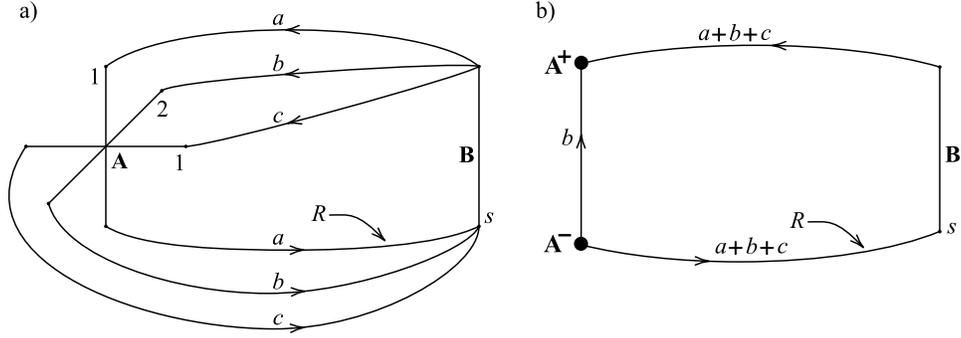}
\caption{The R-R diagram of $R$ and its hybrid diagram.}
\label{DPCFig22es}
\end{figure}
\begin{figure}[tbp]
\centering
\includegraphics[width = 1\textwidth]{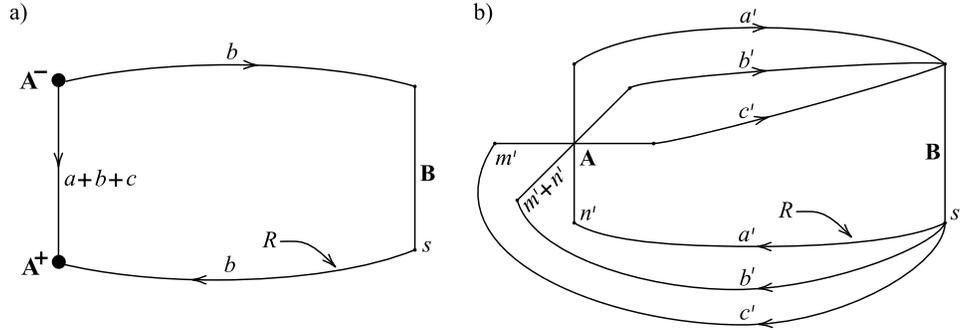}
\caption{The hybrid diagram of $R$ and its R-R diagram after performing a change
of cutting disks inducing an automorphism of $\pi_1(H)$ that takes $A \mapsto AB^{-s}$.}
\label{DPCFig22fs}
\end{figure}

\begin{lem}
\label{may assume (m,n) ?(1,1)}
Suppose $R$ and $\beta$ are simple closed curves with an R-R diagram of the form shown in Figure~\emph{\ref{DPCFig21as}} with $(m,n)=(\pm1,\pm1)$ with the double signs in same order, $a,b,c>0$, and $|s| > 1$. Then there is a change of cutting disks of $H$ such that $R$ and $\beta$ have an R-R diagram of the form shown in Figure~\emph{\ref{DPCFig21as}} with $m , n \neq 0, (m,n)\neq(\pm1,\pm1)$, $\gcd(m,n) = 1$, and $|s| > 1$.
\end{lem}

\begin{proof}
Without loss of generality, we may assume that $(m,n)=(1,1)$ and $s>1$ in the R-R diagram of Figure~\ref{DPCFig21as}. We use the argument of hybrid diagrams introduced in Section~\ref{hybrid diagrams}. The R-R diagram of $R$ and its hybrid diagram are illustrated in Figure~\ref{DPCFig22es}. In its hybrid diagram, we drag the vertex $A^-$
together with the edges meeting the vertex $A^{-}$ over the $s$-connection on the $B$-handle. This performance
corresponds to a change of cutting disks inducing an automorphism of $\pi_1(H)$ that takes $A \mapsto AB^{-s}$ and leaves $B$ fixed.
The resulting hybrid diagram of $R$ is depicted in Figure~\ref{DPCFig22fs}a, where there is only one band of connections labeled by $s$ in the $B$-handle and $R$ intersects the cutting disk $D_A$ positively in the $A$-handle. Therefore the corresponding R-R diagram of $R$ must be the form given in Figure~\ref{DPCFig22fs}b, where $m', n'>0$ and $a'+b'+c'=b$.
We will show that $(m', n')\neq (1,1)$ to complete the proof.

In the R-R diagram of Figure~\ref{DPCFig22es}a, if we read $R$ from the outermost edge of the $c$ parallel edges entering the $1$-connection in the $A$-handle,  ,
then we can see that $R$ has subword $(AB^s)^iA^2B^s$ with $i>1$. Under the automorphism $A \mapsto AB^{-s}$, the subword $(AB^s)^iA^2B^s$ is sent to $A^{i+1}B^{-s}A$.
This implies that the label $i+1(>2)$ appears in the $A$-handle. Thus $R$ has a band of connections with label greater than $2$ in the $A$-handle, which implies that $(m', n')\neq (1,1)$.
\end{proof}

\begin{lem}
\label{may assume m,n not zero}
Suppose $R$ and $\beta$ are simple closed curves with an R-R diagram of the form shown in Figure~\emph{\ref{DPCFig21as}} with either $m=0$ or $n=0$, $a,b,c>0$, and $|s| > 1$. Then there is a change of cutting disks of $H$ such that $R$ and $\beta$ have an R-R diagram of the form shown in Figure~\emph{\ref{DPCFig21as}} with $m , n \neq 0, (m,n)\neq(\pm1,\pm1)$, $\gcd(m,n) = 1$.
\end{lem}

\begin{proof}
Suppose one of $m$ and $n$ is $0$. Then without loss of generality,
we may assume that $n=0$, because both cases have R-R diagrams of the same form. Then $m=1$ and the Heegaard diagram of $R$ has a cut-vertex. Note that
$m$ and $n$ are not both $0$, otherwise $R$ represents $B^{ts}$ in $\pi_1(H)$ with $t>1$.
This is impossible since $H_1(H[R])$ is torsion-free.

\begin{figure}[tbp]
\centering
\includegraphics[width = 1\textwidth]{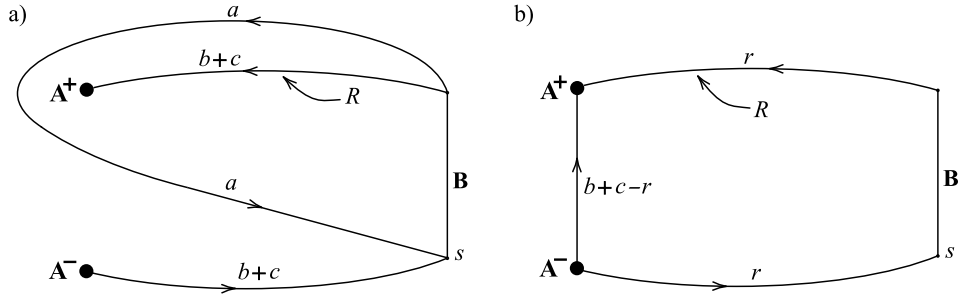}
\caption{The Hybrid diagram and
the hybrid diagram after performing a change of cutting disks inducing an automorphism of $\pi_1(H)$ that takes $A \mapsto AB^{-(\rho+1)s}$. }
\label{PSFFig3ad-4}
\end{figure}

Since the Heegaard diagram of $R$ has a cut-vertex, as did in Lemma~\ref{may assume (m,n) ?(1,1)}, we use the argument of hybrid diagrams. The hybrid diagram of $R$ is illustrated in Figure~\ref{PSFFig3ad-4}a.
Let $a=\rho (b+c)+r$, where $\rho\geq 0$ and $0\leq r<b+c$. If $r=0$, then $a=\rho (b+c)$ and thus it follows from the R-R diagram of $R$
that $R$ represents $(AB^{(\rho+1)s})^{b+c}$ in $\pi_1(H)$, a contradiction. Therefore $r>0$.
Now in the hybrid diagram, we drag $\rho+1$ times the vertex $A^-$
together with edges meeting the vertex $A^{-}$ over the $s$-connection on the $B$-handle. This performance
corresponds to a change of cutting disks inducing an automorphism of $\pi_1(H)$ that takes $A \mapsto AB^{-(\rho+1)s}$.
The resulting hybrid diagram of $R$ is depicted in Figure~\ref{PSFFig3ad-4}b, where there is only one band of connections labeled by $s$ in the $B$-handle.
Also in the $A$-handle, $R$ intersects the cutting disk $D_A$ positively, and since $b+c-r>0$, there must be a band of connections whose label is greater than $1$.
Therefore the corresponding R-R diagram of $R$ has the form of Figure~\ref{DPCFig21as} with $m , n \neq 0$ and $\gcd(m,n) = 1$. If $(m,n)\neq(\pm1,\pm1)$, then by applying
Lemma~\ref{may assume (m,n) ?(1,1)}, we obtain an R-R diagram of $R$ with the form of Figure~\ref{DPCFig21as} with $m , n \neq 0, (m,n)\neq(\pm1,\pm1)$, $\gcd(m,n) = 1$, as desired.
\end{proof}

\begin{rem}
According to Lemmas~\ref{may assume (m,n) ?(1,1)} and \ref{may assume m,n not zero},
the resulting R-R diagram of $R$ obtained by performing a change of cutting disks from
an R-R diagram of $R$ with the form shown in Figure~\ref{DPCFig21as} with $a,b,c>0$, and $|s| > 1$ such that $(m,n)=(\pm1,\pm1)$ with the double signs in same order or either $m=0$ or $n=0$, has the form of Figure~\ref{DPCFig21as} with $m , n \neq 0, (m,n)\neq(\pm1,\pm1)$. We remark that in the resulting R-R diagram of $R$, one or two weights of $a, b,$ and $c$ might be $0$. In that case, as defined in the next section, $R$ becomes a torus or cable knot relator, which means that $H[R]$ is the exterior of a torus knot or a cable knot.
\end{rem}
\smallskip

\subsection{R-R diagrams of the exteriors of torus knots and tunnel-number-one cables of torus knots in $S^3$}\hfill
\label{S: R-R diagrams of torus and cable knots}
\smallskip

Theorem~\ref{at least one of a, b, c is zero} implies that if $R$ has only one band of connections with label greater than 1 on one handle, then
in order for $H[R]$ to embed in $S^3$, $R$ must have at most two bands of connections on the other handle. In this subsection, we discuss this case, i.e., the case where $R$ has an R-R diagram of the form shown in Figure~\ref{DPCFig23as} and determine which values of the parameters give diagrams of embeddings in $S^3$.

\begin{lem}\label{gcd(a,b)=1}
Suppose $R$ is a simple closed curve with an R-R diagram of the form shown in Figure~\emph{\ref{DPCFig23as}}.
If $a$ and $b$ are positive, then $\gcd(a,b) = 1$.
\end{lem}

\begin{proof}
Observe that the simple closed curves $M$ and $\beta$ in Figure~\ref{DPCFig23as}, intersect transversely at a single point. Then $R$ lies in a once-punctured torus $F$ in $\Sigma$, which is disjoint from $M \cup \beta$. Since $R$ represents a nontrivial homology class in the integral homology group $H_1(F)$, the result follows from the well-known characterization of the homology classes in $H_1(F)$ which are represented by simple closed curves.
\end{proof}

\begin{figure}[tbp]
\centering
\includegraphics[width = 0.6\textwidth]{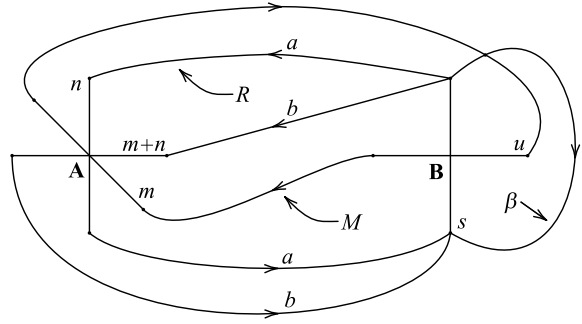}
\caption{An R-R diagram of $R$ where $R$ has only one band of connections with the label greater than 1 on one handle and has at most two bands of connections on the other handle.}
\label{DPCFig23as}
\end{figure}

We first consider a positive Heegaard diagram of $R$ with no cut-vertex which underlies the R-R diagram in Figure~\ref{DPCFig23as} and
figure out which conditions on the parameters in the R-R diagram of $R$ are required for $H[R]$ to embed in $S^3$.

\begin{lem}
\label{diagrams of curves that embed when c = 0}
Suppose $R$ is a simple closed curve with an R-R diagram of the form shown in Figure~\emph{\ref{DPCFig23as}} with $m, n > 0$, $\gcd(m,n) = 1$, and $s > 1$. In addition, suppose that if $a+b = 1$, then $(a, b)$ = $(1, 0)$ and $n > 1$.

If $H[R]$ embeds in $S^3$, then the curve $M$ in Figure~\emph{\ref{DPCFig23as}} represents the meridian of $H[R]$, $0 < u < s$, $\gcd(s,u) = 1$, and there exists $\delta = \pm 1$ such that one of the following three sets of conditions is satisfied:
\begin{enumerate}
  \item $a + b = 1$ and, \emph{(}assuming without loss that $(a,b)$ = $(1,0)$\emph{)}, $sm - un = \delta$.\\
  \item $a,b > 0$, $m = 1$, and $s(a+b) - u[(a+b)n +b] = \delta$.\\
  \item $a,b > 0$, $n> m > 1$, $u = 1$, and $ms(a+b) - [(a+b)n +bm] = \delta$.\\
\end{enumerate}
Conversely, if one of these three sets of conditions is satisfied, then $H[M,R]$ is $S^3$.
\end{lem}

\begin{proof}
The Heegaard diagram of $R$ underlying the R-R diagram in Figure~\ref{DPCFig23as} is positive, connected, and has no cut-vertex. It follows from Theorem~\ref{waves provide meridians} that if $H[R]$ embeds in $S^3$, then a representative of the meridian of $H[R]$ can be obtained by surgery on $R$ along a horizontal wave $\omega_h$. The curve $M$ in Figure~\ref{DPCFig23as} is obtained by such a
surgery on $R$, and so, if $H[R]$ embeds in $S^3$, $M$ represents the meridian of $H[R]$. From Figure~\ref{DPCFig23as}, we see that $M = A^mB^u$ in $\pi_1(H)$ with $0 < u < s$ and $\gcd(s,u) = 1$. Next, consideration of the possible values of $m$ and $u$ leads to the three cases of the lemma.

First, suppose $m,u > 1$. Then by Lemma 2.2 of \cite{D03} or Theorem 4.2 of \cite{K20}, $M$ is a Seifert fiber relator, i.e., $H[M]$ is Seifert-fibered with two exceptional fibers of indexes $m$ and $u$. In addition, Figure~\ref{DPCFig23as} shows $R$ is a simple closed curve in $\partial H[M]$, which intersects a regular fiber of $H[M]$ $a+b$ times. This implies $H[R,M]$ will be Seifert-fibered with three exceptional fibers, and not homeomorphic to $S^3$, unless $a+b = 1$.
When $a+b = 1$, we may assume $a = 1$, $b = 0$ and $R = A^nB^s$. It follows that $H[M,R]$ will be $S^3$ if and only if $sm - un = \delta$.
This gives case (1) of the lemma. (Assuming $n > 1$ in this case avoids the trivial situation in which $H[R]$ is a solid torus.)

Next, suppose $m = 1$. Then $M$ is primitive in $H$ and $H[R,M]$ will be $S^3$ if and only if $H[R,M]$ is an integral homology sphere. Clearly, this happens if and only if $s(a+b) - u[(a+b)n +b] = \delta$, which gives case (2) of the lemma.

Finally, suppose $m > 1$ and $u = 1$. Then, again, $M$ is primitive in $H$ and $H[M,R]$ will be $S^3$ if and only if $H[M,R]$ is an integral homology sphere. Examination of Figure~\ref{DPCFig23as} shows this happens if and only if $ms(a+b) - [(a+b)n +bm] = \delta$.  This gives case (3) of the lemma except for the condition that $n > m$.
However, the equation $ms(a+b) - [(a+b)n +bm] = \delta$, where $\delta=\pm1$ implies that $(n-m)(a+b)=m(s-1)(a+b)-bm-\delta$.
Since $m, s>1$, the right-hand side is positive. Therefore $n >m$ as desired.
\end{proof}

\begin{figure}[tbp]
\centering
\includegraphics[width = 1\textwidth]{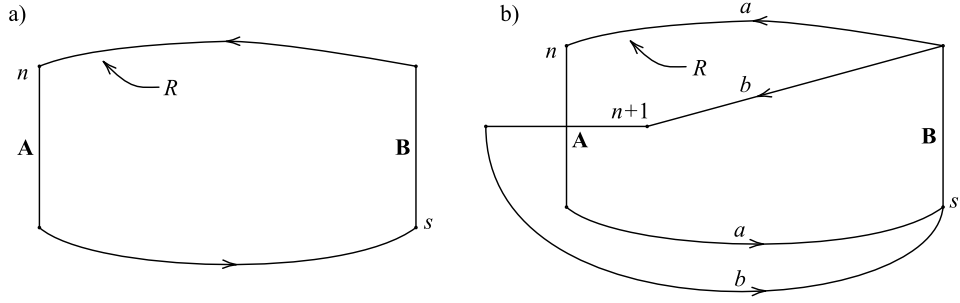}
\caption{Suppose $R$ is a simple closed curve on the boundary of a genus two handlebody $H$ such that $R$ has an R-R diagram
as in this figure, with $n, s>1$, $a,b>0$, $\gcd(a,b)=1$. Then $H[R]$ is the exterior of an $(n, s)$ torus knot (Figure~\ref{PSFFig3aa}a) or
an $((a+b)n+b, s)$ torus knot (Figure~\ref{PSFFig3aa}b).}
\label{PSFFig3aa}
\end{figure}

To prove the next proposition, we need the following notations and lemma.

\noindent\textit{Notation}. If $U = (a,b)$ is an element of $\mathbb{Z} \oplus \mathbb{Z}$, let $U^{\perp}$ denote the element $(-b,a)$ of $\mathbb{Z} \oplus \mathbb{Z}$, and let
 `$\circ$' denote the usual \emph{inner product} or \emph{dot product} of vectors.

\begin{lem}
\label{dot product}
Let $U = (a,b)$, $V = (c,d)$ and $W = (e,f)$ be three elements of $\mathbb{Z} \oplus \mathbb{Z}$ such that
$ad-bc = \pm1$. If $W$ is expressed as a linear combination of $U$ and $V$, say $W = xU + yV$, then $y = \pm (U^{\perp} \circ W)$.
\end{lem}

\begin{proof}
We take an inner product by $U^{\perp}$ on both sides of $W = xU + yV$. Then since $U^{\perp}\circ U=0$ and $U^{\perp}\circ V=\pm1$, the result follows.
\end{proof}

\begin{prop}\label{torus or cable knot exterior}
Suppose the hypothesis of Lemma~\emph{\ref{diagrams of curves that embed when c = 0}}.
Then the curve $\beta$ in Figure~\emph{\ref{DPCFig23as}} intersects the meridian of $H[R]$ once and two parallel copies of the curve $\beta$ bound a separating essential annulus $\mathcal{A}$ in $H[R]$, and:
\begin{enumerate}
\item If Case \emph{(1)} of Lemma~\emph{\ref{diagrams of curves that embed when c = 0}} holds, then  $H[R]$ is the exterior of an $(n,s)$ torus knot in $S^3$, and $\beta$ represents a regular fiber in the Seifert fibration of $H[R]$.

\item If Case \emph{(2)} of Lemma~\emph{\ref{diagrams of curves that embed when c = 0}} holds, then $H[R]$ is the exterior of an \\$((a+b)n +bm,s)$ torus knot in $S^3$, and $\beta$ represents a regular fiber in the Seifert fibration of $H[R]$.

\item If Case \emph{(3)} of Lemma~\emph{\ref{diagrams of curves that embed when c = 0}} holds, then  $H[R]$ is the exterior of an \\$(ms(a+b) \pm1,s)$ cable about an $(a+b,m)$ torus knot in $S^3$ and $\mathcal{A}$ is the cabling annulus in $H[R]$.
\end{enumerate}
\end{prop}

\begin{proof}
It is obvious from Figure~\ref{DPCFig23as} that the curve $\beta$ intersects the meridian of $H[R]$ once.

The diagrams of $R$ in Cases $(1)$ and $(2)$ are shown in Figures~\ref{PSFFig3aa}a and \ref{PSFFig3aa}b respectively.
Two parallel copies of the curve $\beta$ bound a separating essential annulus $\mathcal{A}$ in $H$ and thus in $H[R]$.
By Lemma 2.2 of \cite{D03} or Theorem 4.2 of \cite{K20}, $H[R]$ is a Seifert-fibered spaces over $D^2$ with two exceptional fibers of indexes $n$ and $s$ in
Figure~\ref{PSFFig3aa}a, and $(a+b)n+b$ and $s$ in Figure~\ref{PSFFig3aa}b, and $\beta$ represents a regular fiber in the Seifert fibration of $H[R]$. However, $H[R]$ is the exterior of some knot in $S^3$. It is known that torus knot exteriors are Seifert-fibered, and by \cite{MOS71}, are the only knot exteriors in $S^3$ which are Seifert-fibered. Therefore, $H[R]$ is the exterior of a torus knot, i.e., an $(n, s)$ torus knot (Figure~\ref{PSFFig3aa}a) or an $((a+b)n+b, s)$ torus knot (Figure~\ref{PSFFig3aa}b).

\begin{figure}[tbp]
\centering
\includegraphics[width = 0.6\textwidth]{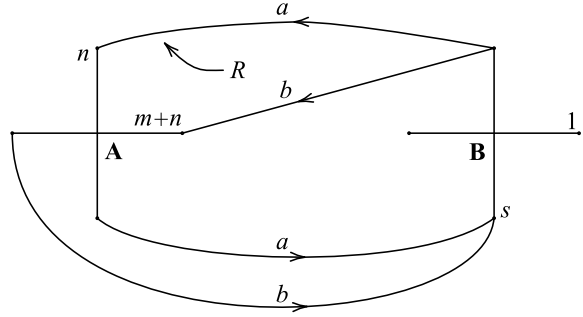}
\caption{Suppose $R$ is simple closed curve on the boundary of a genus two handlebody $H$ such that $R$ has an R-R diagram as in this figure, with $m, n, s > 1$, $a, b > 0$, $\gcd(a,b)$ = $\gcd(m,n) = 1$, and $n(a+b) + bm$ = $sm(a+b) + \delta$ with $\delta = \pm 1$. Then $H[R]$ is the exterior of an $(ms(a+b) \pm1,s)$ cable about an $(a+b,m)$ torus knot.}
\label{PSFFig3ab}
\end{figure}

Assume that Case $(3)$ holds. The diagram of $R$ in this case is given in Figure~\ref{PSFFig3ab}. As in Cases $(1)$ and $(2)$, two parallel copies of the curve $\beta$ bound a separating essential annulus $\mathcal{A}$ in $H$ and thus in $H[R]$.

Let $\mathcal{D}$ denote the R-R diagram of Figure~\ref{PSFFig3ab}. Suppose that $D_A$ and $D_B$ are cutting disks of $H$ underlying the A-handle and B-handle respectively of $\mathcal{D}$.

Cutting $H$ open along $\mathcal{A}$ yields a solid torus $V$ and a genus two handlebody $W$. Observe that $R$ lies in $\partial W$, and $W$ has cutting disks $D_A$ and $D_C$, where $D_C$ is one of the outermost subdisks that $\mathcal{A}$ cuts out of $D_B$.

\begin{claim}\label{$W[R]$ is the exterior of a $(m, a+b)$ torus knot}
$W[R]$ is the exterior of an $(m, a+b)$ torus knot.
\end{claim}
\begin{proof}
Consider the Heegaard diagram of $R$ with respect to $\{D_A, D_C\}$ of $W$. Since $R$ intersects $D_C$ once before and after intersecting $D_A$, $R$ has an R-R diagram as shown in Figure~\ref{PSFFig3ac}a which the Heegaard diagram underlies. Note that the diagram of $R$ in Figure~\ref{PSFFig3ac}a can be obtained from the diagram of $R$ in Figure~\ref{PSFFig3ab} with $s$ set equal to $1$.

Let $A$ and $C$ denote generators of $\pi_1(W)$ dual to the cutting disks $D_A$ and $D_C$ respectively of $W$. For the weights $a$ and $b$ in the R-R diagram of $R$, since gcd$(a,b)=1$, we can let $a=\rho b +r$, where $\rho\geq 0$ and $0\leq r<b$ (if $r=0$, then $\rho>0$ and $b=1$).

Now we record $R$ by starting the $b$ parallel arcs entering into the $(m+n)$-connection in the $A$-handle. It follows from the R-R diagram
that $R$ is the product of two subwords $A^{m+n}(CA^{n})^\rho C$ and $A^{m+n}(CA^{n})^{\rho+1}C$ with $|A^{m+n}(CA^{n})^\rho C|=b-r$ and $|A^{m+n}(CA^{n})^{\rho+1}C|=r$. There is a change of cutting disks of the handlebody $W$ underlying the diagram, which induces an automorphism of $\pi_1(H)$ that takes $C \mapsto CA^{-n}$ and leaves $A$ fixed. Then by this change of cutting disks, $A^{m+n}(CA^{n})^\rho C$ and $A^{m+n}(CA^{n})^{\rho+1}C$ are sent to $A^mC^{\rho+1}$ and $A^mC^{\rho+2}$. Therefore the resulting Heegaard diagram of $R$ realizes a new R-R diagram of the form in Figure~\ref{PSFFig3ac}b, where the positions of the $A$- and $C$-handles are switched. The R-R diagram of $R$ in Figure~\ref{PSFFig3ac}b has the same form as in Figure~\ref{PSFFig3aa}b implying that $W[R]$ is the exterior of an $(a+b, m)$ torus knot in $S^3$. This completes the proof of the claim.
\end{proof}

\begin{figure}[tbp]
\centering
\includegraphics[width = 1\textwidth]{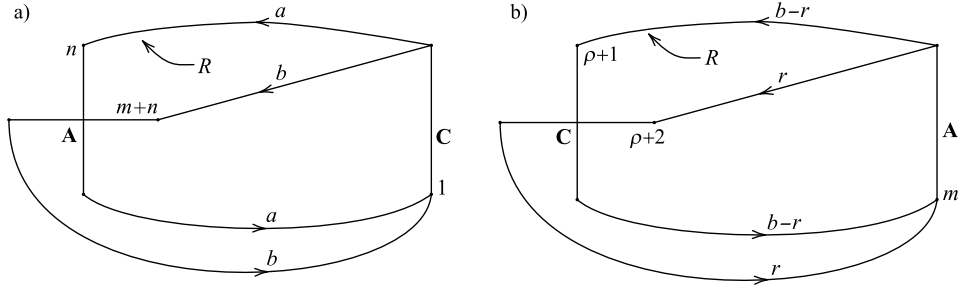}
\caption{An R-R diagram of $R$ with $s=1$ and new R-R diagram of $R$ after the change of cutting disks.}
\label{PSFFig3ac}
\end{figure}

By the claim, $W[R]$ is the exterior of an $(a+b, m)$ torus knot in $S^3$, therefore $H[R]$ is a cable about an $(a+b, m)$ torus knot in $S^3$ with
$\mathcal{A}$ the cabling annulus in $H[R]$.

It remains to compute the cabling coordinates. Let $l$ be the longitude of $W[R]$. Then, since $\beta$ is a boundary component of the cabling annulus $\mathcal{A}$, the cabling coordinates are given, up to sign, by $\Delta(\beta,l)$ and $\Delta(\beta,\partial D_B)$. It is obvious that $\Delta(\beta,\partial D_B) = s$.

\begin{claim}\label{cabling coordinates}
$\Delta(\beta,l)=(a+b)n+bm$.
\end{claim}

\begin{proof}
Let $M$ be a meridian of $W[R]$. Then $H_1(W[R])(=\mathbb{Z})$ is generated by $[M]$ and $[\beta]=\Delta(\beta,l)[M]$ in $H_1(W[R])$. Since $W[R,M]\cong S^3$, $\{[R], [M]\}$ is a basis of $H_1(W)=\mathbb{Z}\oplus\mathbb{Z}$. If we let $[\beta]=x[R]+y[M]$ in $H_1(W)$, then since $W[R]$ is obtained from $W$ by gluing a 2-handle along $R$, $y=\Delta(\beta,l)$.
In order to compute $y$, we apply for Lemma~\ref{dot product}. It follows from Figure~\ref{PSFFig3ac}a that $[\vec{R}] = ((a+b)n+bm, a+b)$ and $[\vec{\beta}] = (0,1)$ in $H_1(W)$. Then Lemma~\ref{dot product} implies
$$y = | [\vec{R}]^{\perp}  \circ [\vec{\beta}] | = |((a+b)n+bm, a+b)^{\perp} \circ (0,1)| = (a+b)n+bm.$$
This completes the claim.
\end{proof}

Therefore, $H[R]$ is an $((a+b)n+bm, s)$ cable about an $(a+b, m)$ torus knot. Since $ms(a+b) - [(a+b)n +bm] = \pm1$,  $H[R]$ is an $(ms(a+b)\pm1, s)$ cable about an $(a+b, m)$ torus knot, as desired.
\end{proof}

We note that in Cases (1), (2), and (3), the curve $R$ is the boundary of the co-core of a 1-handle regular neighborhood of some unknotting tunnel
of a torus knot for (1) and (2) and of a cable of a torus knot for (3).
However since there are the classifications of unknotting tunnels
of torus knots and cables of torus knots up to homeomorphism in \cite{BRZ88}
and \cite{MS91} respectively, the converse of Proposition~\ref{torus or cable knot exterior} is also true as the following propositions verify.

\begin{prop}\label{the converse true for torus knots}
Suppose $R$ is a simple closed curve on the boundary of a genus two handlebody $H$
such that $H[R]$ is homeomorphic to the exterior of a torus knot in $S^3$. Then
$R$ has an R-R diagram with the form in Figures~\emph{\ref{PSFFig3aa}a} or \emph{\ref{PSFFig3aa}b} with the parameters satisfying the conditions
of Cases $(1)$ and $(2)$ respectively in Lemma~\emph{\ref{diagrams of curves that embed when c = 0}}.
\end{prop}
\begin{proof}
This follows from Theorem 3.2 of \cite{K20}, which provides a classification of R-R diagrams of $R$ such that $H[R]$ is a Seifert-fibered space over the disk with two exceptional fibers.

The idea of its proof is that since $H[R]$ is homeomorphic to the exterior of a torus knot in $S^3$, $H[R]$ is
a Seifert-fibered space $\mathcal{M}$ over the disk with two exceptional fibers and by the definition of a 2-handle addition, $H[R]$ induces
a genus two Heegaard decomposition of $\mathcal{M}$. The result of \cite{BRZ88} shows that
there are three genus two Heegaard decompositions of $\mathcal{M}$ up to homeomorphism. One of them gives an R-R diagram of Figure~\ref{PSFFig3aa}a, and the other two are symmetric each other so that they boil down to an R-R diagram of Figure~\ref{PSFFig3aa}b.
\end{proof}

\begin{prop}\label{the converse true for cable knots}
Suppose $R$ is a simple closed curve on the boundary of a genus two handlebody $H$
such that $H[R]$ is homeomorphic to the exterior of a cable of a torus knot in $S^3$. Then
$R$ has an R-R diagram with the form in Figure~\emph{\ref{PSFFig3ab}}
with the parameters satisfying the conditions of Case $(3)$ in Lemma~\emph{\ref{diagrams of curves that embed when c = 0}}.
\end{prop}

\begin{proof}
Suppose $k$ is a cable of a torus knot in $S^3$ whose exterior is homeomorphic to $H[R]$.
Note that $k$ is a tunnel-number-one knot such that $R$ is the boundary of a cocore of the 1-handle regular neighborhood of a tunnel. In other words, the curve $R$ corresponds to a tunnel of $k$. By the result of \cite{MS91} $k$ is an $(spq+\delta, s)$ cable of a $(p,q)$ torus knot with $\delta=\pm1$ and has two unknotting tunnels up
to homeomorphism. Here, we may assume that $p,q>1$ and $s>1$.

\begin{claim} \label{no homeo from R1 to R2}
There exist two simple closed curves $R_1$ and $R_2$ on $\partial H$ such that
\begin{enumerate}
\item $H[R_1]$ and $H[R_2]$ are homeomorphic to the exterior of the $(spq+\delta, s)$ cable of a $(p,q)$ torus knot $k$,
\item $R_1$ and $R_2$ have R-R diagrams of the form shown in Figure~\emph{\ref{PSFFig3ab}}, and
\item there is no homeomorphism from $H$ onto itself sending $R_1$ to $R_2$.
\end{enumerate}
\end{claim}

\begin{proof}
For the values $p, q, s$, and $\delta$, there are two sets of integers $(m_1, n_1, a_1, b_1)$ and $(m_2, n_2, a_2, b_2)$ with $m_i, n_i>1$ and $a_i,b_i>0$ for $i=1,2$ such that
\[
\begin{split}
&p=a_1+b_1, q=m_1, n_1(a_1+b_1) + b_1m_1 = sm_1(a_1+b_1) + \delta, \hspace{0.2cm}\text{and}\\
&p=m_2, q=a_2+b_2, n_2(a_2+b_2) + b_2m_2 = sm_2(a_2+b_2) + \delta.
\end{split}
\]

\noindent For the set of integers $(m_i, n_i, a_i, b_i)$ for $i=1,2$, we consider a simple closed curve $R_i$ which has an R-R diagram of the form in Figure~\ref{PSFFig3ab} with $(m,n,a,b)=(m_i,n_i,a_i,b_i)$. Then by Proposition~\ref{torus or cable knot exterior},
$H[R_i]$ is homeomorphic to the exterior of $k$.

To complete the proof of the claim, it remains to show that there is no homeomorphism from $H$ onto itself sending $R_1$ to $R_2$. We can observe from an R-R diagram of the form in Figure~\ref{PSFFig3ab} with $(m,n,a,b)=(m_i,n_i,a_i,b_i)$ that $R_1$ and $R_2$ have the Heegaard diagrams of the form shown in Figures~\ref{PSFFig3af}a and ~\ref{PSFFig3af}b respectively.
Since $p,q>1$ and $s>1$, it follows that $(s-1)pq+\delta>p$ and $(s-1)p\geq p$ for $R_1$, and $(s-1)pq+\delta>q$ and $(s-1)q\geq q$ for $R_2$. This implies that $R_1$ and $R_2$ have minimal lengths which are distinct. Therefore by the result of \cite{W36}, there is no homeomorphism from $H$ onto itself sending $R_1$ to $R_2$.

\begin{figure}[tbp]
\centering
\includegraphics[width = 0.9\textwidth]{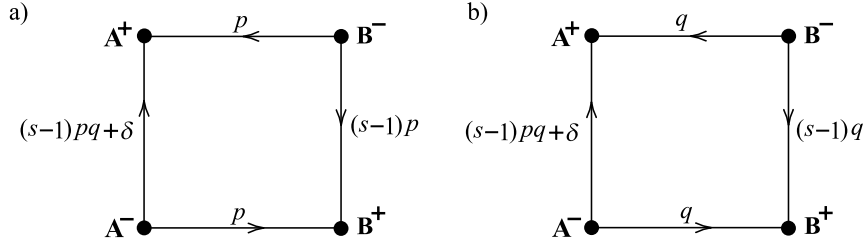}
\caption{The Heegaard diagrams of $R_1$ and $R_2$ underlying the R-R diagrams in Figure~\ref{PSFFig3ab}.}
\label{PSFFig3af}
\end{figure}

\end{proof}

Since $k$ has two unknotting tunnels up to homeomorphism, Claim~\ref{no homeo from R1 to R2} implies that $R_1$ and $R_2$ represent the two tunnels of $k$. Therefore
if $R$ is a simple closed curve on $\partial H$ such that $H[R]$ is homeomorphic to the exterior of $k$, then there exists a homeomorphism from $H$ onto itself
sending $R$ to either $R_1$ or $R_2$, whose corresponding R-R diagrams are of the form in Figure~\ref{PSFFig3ab}. Therefore, $R$ has an R-R diagram with the form in Figure~\ref{PSFFig3ab}.
\end{proof}

Now we give some terminology for R-R diagrams of $R$ with the form in Figures~\ref{PSFFig3aa} and \ref{PSFFig3ab}.

\begin{defn}
[\textbf{Torus or cable knot relators}]
If a simple closed curve $R$ in the boundary of a genus two handlebody $H$ has an R-R diagram of the form shown in Figure~\ref{PSFFig3aa} and thus $H[R]$ is the exterior of a torus knot, then $R$ is said to be \textit{a torus knot relator}. In particular, if $R$ has the diagram in Figure~\ref{PSFFig3aa}a (Figure~\ref{PSFFig3aa}b, resp.), then $R$ is called \textit{a rectangular(non-rectangular, resp.) torus knot relator}.

If a simple closed curve $R$ has an diagram of the form shown in Figure~\ref{PSFFig3ab} and thus $H[R]$ is the exterior of a cable of a torus knot,  then $R$ is called \textit{a cable knot relator}.
\end{defn}
\smallskip

\subsection{Main results: R-R diagrams of tunnel-number-one knot exteriors disjoint from proper power curves}\hfill
\label{R-R diagrams of $R$ disjoint from a proper power curve}
\smallskip

Now we prove Theorem~\ref{main theorem 1}, which is a combination of Proposition~\ref{torus or cable knot exterior} and the following theorem.
Also for convenience, we put a copy of Figure~\ref{DPCFig21b-1} in
Figure~\ref{DPCFig21b}.

\begin{thm}\label{together with a proper power curve}
Suppose $R$ and $\beta$ are disjoint simple closed curves in the boundary of a genus two handlebody $H$ such that $H[R]$ embeds in $S^3$ and $\beta$ is a proper power curve.
Then $R$ and $\beta$ have an R-R diagram with the form shown either in Figure~\emph{\ref{DPCFig21b}a} with $s>1$ or Figure~\emph{\ref{DPCFig21b}b} with the set of parameters $(a,b,m,n,s)$ satisfying the hypothesis of Lemma~\emph{\ref{diagrams of curves that embed when c = 0}} so that $(a,b,m,n,s)$ has the condition \emph{(1), (2),} or \emph{(3)} in Lemma~\emph{\ref{diagrams of curves that embed when c = 0}}.
\end{thm}

\begin{figure}[tbp]
\centering
\includegraphics[width = 0.7\textwidth]{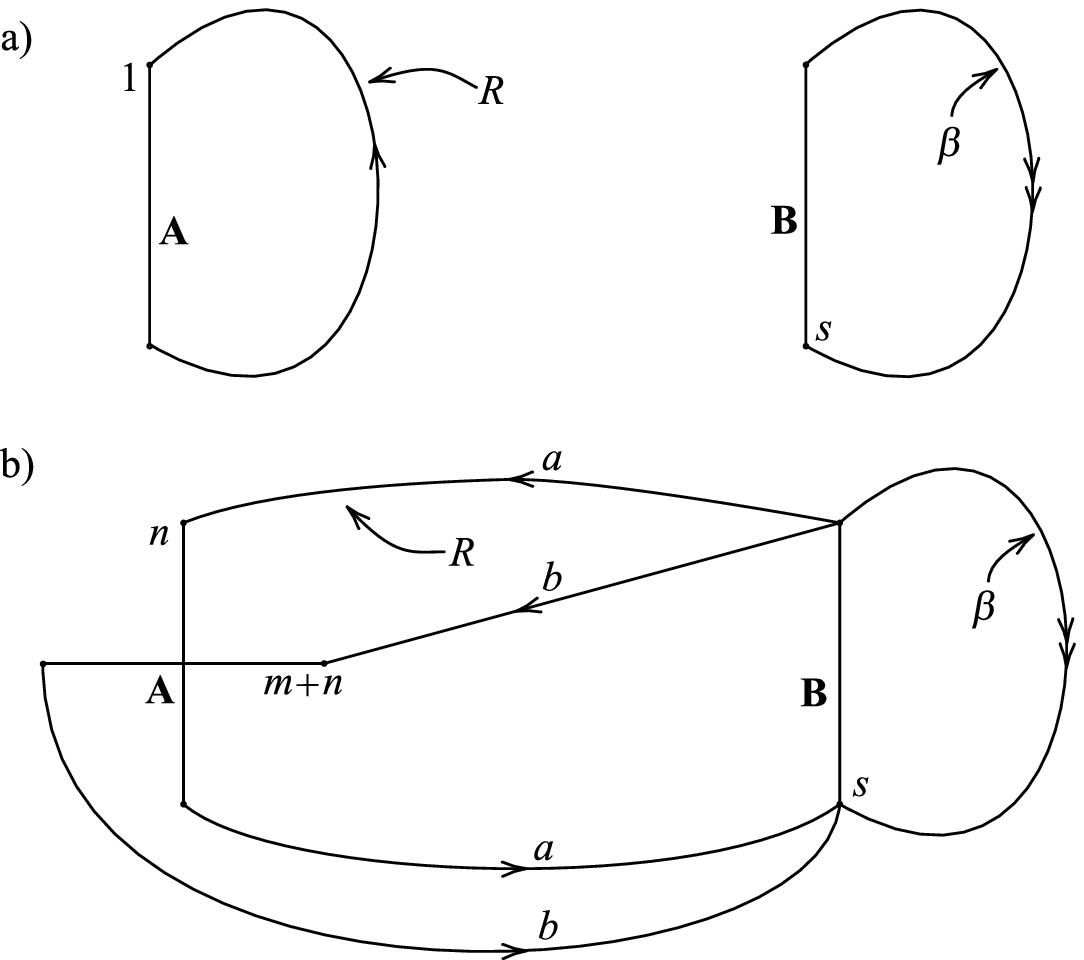}
\caption{All possible R-R diagrams of disjoint simple closed curves $R$ and $\beta$ up to homeomorphism on $H$ such that $H[R]$ embeds in $S^3$ and $\beta$ is a proper power curve.}
\label{DPCFig21b}
\end{figure}

\begin{proof}
By the definition of a proper power curve, there exist an essential separating disk $D_C$
disjoint from $\beta$ and a complete set of cutting disks $\{D_A, D_B\}$ of $H$ such that
$D_C\cap(D_A\cup D_B)=\varnothing$, $\beta\cap D_A=\varnothing$, and $|\beta\cap D_B|=s(>1)$. This implies that with respect to the set of cutting disks $\{D_A, D_B\}$, $\beta$ has an R-R diagram of the form in Figure~\ref{DPCFig21b}.

Suppose a simple closed curve $R$ is added to the R-R diagram of $\beta$. If $R$ has no connections in the $B$-handle, then $R$ has only one connection in the $A$-handle.
Since $H[R]$ embeds in $S^3$, the label of the connection should be $\pm1$. Therefore in this case $R$ and $\beta$ have an R-R diagram which has the form shown in Figure~\ref{DPCFig21b}a up to equivalence.

Now suppose $R$ has connections in both $A$- and $B$-handles. Since $R$ is disjoint from $\beta$, $R$ must have an R-R diagram of the form
shown in Figure~\ref{DPCFig21as-4}, where $m, n\in \mathbb{Z}$ with gcd$(m, n)=1$.
Then $R$ must be positive, otherwise as in the proof of Theorem~\ref{at least one of a, b, c is zero}, one meridian representative of $H[R]$, which is obtained by surgery on $R$ along a vertical wave, is isotopic to $\beta$, a contradiction.

\begin{figure}[tbp]
\centering
\includegraphics[width = 0.7\textwidth]{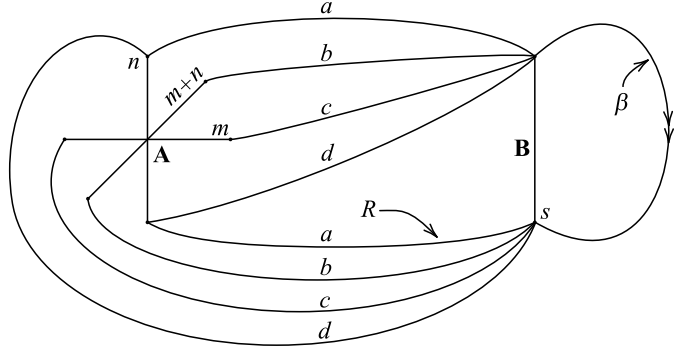}
\caption{An R-R diagram of $R$ which is disjoint from a proper power curve $\beta$ with $[\beta]=B^s$ in $\pi_1(H)$.}
\label{DPCFig21as-4}
\end{figure}

The positivity of $R$ implies that either $a=0$ or $d=0$. We may assume that $d=0$ and thus $R$ has an R-R diagram of the form in Figure~\ref{DPCFig21as} because if $a=0$, then
$R$ has an R-R diagram of the same form in Figure~\ref{DPCFig21as}.

Since $H[R]$ embeds in $S^3$, by Theorem~\ref{at least one of a, b, c is zero} and Lemmas~\ref{may assume (m,n) ?(1,1)} and \ref{may assume m,n not zero}, we may assume that the R-R diagram of $R$ and $\beta$ has the form shown in Figure~\ref{DPCFig21b}b with $ab\geq 0$ and $n, m+n\geq 0$. The condition $n, m+n\geq 0$ follows from the positivity of $R$ and the equivalence of R-R diagrams. To complete the proof, we will show that the R-R diagram of $R$ and $\beta$ in Figure~\ref{DPCFig21b}b with $ab\geq 0$ and $n, m+n\geq 0$
is equivalent to either the R-R diagram of Figure~\ref{DPCFig21b}a or the R-R diagram
of Figure~\ref{DPCFig21b}b with $m,n>0$, in which case the set of parameters $(a,b,m,n,s)$ satisfies the hypothesis of Lemma~\ref{diagrams of curves that embed when c = 0} as desired. Here we note that $m(>0)$ is the parameter obtained by subtracting the label of the vertical connection from the label of the horizontal connection.

First, assume that $ab=0$. We may assume that $b=0$ and thus $a=1$. If $n=0$, then $R=B^s$ in $\pi_1(H)$, a contradiction. If $n>1$, then
this belongs to the case (1) in Lemma~\ref{diagrams of curves that embed when c = 0}. If $n=1$, then $R=AB^s$
and $\beta=B^s$. We perform a change of cutting disk of $H$ inducing an automorphism
of $\pi_1(H)$ which takes $A\mapsto AB^{-s}$. Then $R$ and $\beta$ are carried to $R=A$
and $\beta=B^s$ and it is easy to see that the resulting R-R diagram has the form of Figure~\ref{DPCFig21b}a.

Second, assume that $ab>0$. We divide the argument into two subcases: (i) $n(m+n)>0$
and (ii) $n(m+n)=0$.

(i) Suppose $n(m+n)>0$. If $m=0$, then since $\gcd(m,n) = 1$, $n=1$ and thus $R=AB^{(a+b)s}$, a contradiction. Therefore $m\neq0$, which implies that either $m+n>n>0$ or $n>m+n>0$. If $m+n>n>0$,
then $m>0$ and thus we are done. If $n>m+n>0$, then we perform an orientation-reversing homeomorphism on $H$ which switches the vertical connection labeled by $n$ and the horizontal connection labeled by $m+n$ in the $A$-handle and swaps the arc of weight $a$ and the arc of weight $b$ in the R-R diagram of $R$. Therefore the resulting R-R diagram of $R$ has the vertical connection labeled by $m+n$ and the horizontal connection labeled by $n$ with $n-(m+n)=-m>0$. Therefore the resulting R-R diagram of $R$ belongs to the form of R-R diagram of Figure~\ref{DPCFig21b}b with $m,n>0$.

(ii) Suppose $n(m+n)=0$. Note that both $n$ and $m+n$ cannot be $0$. Therefore without loss of generality we may assume that $n+m=1$ and then $n=0$.
Since $n=0$, there is a $0$-connection in the diagram of $R$. We use the argument of the hybrid diagram.

\begin{figure}[tbp]
\centering
\includegraphics[width = 1\textwidth]{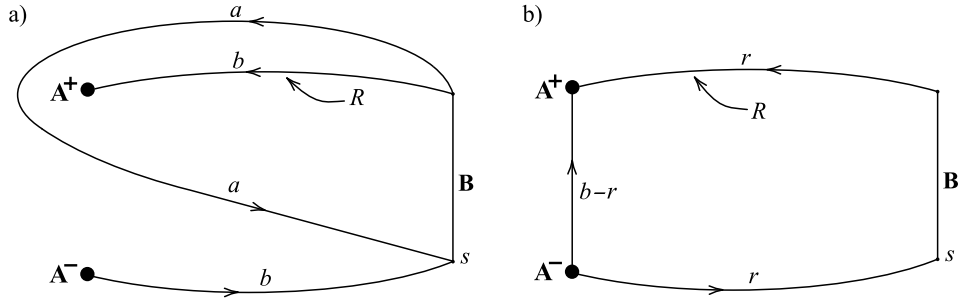}
\caption{The Hybrid diagram corresponding to the R-R diagram in Figure~\ref{DPCFig21b}b with $(m, n)=(1,0)$ and
the hybrid diagram after performing a change of cutting disks inducing an automorphism of $\pi_1(H)$ that takes $A \mapsto AB^{-(\rho+1)s}$. }
\label{PSFFig3ad}
\end{figure}

Its hybrid diagram of $R$ corresponding to the R-R diagram in Figure~\ref{DPCFig21b}b with $(m, n)=(1,0)$ is illustrated in Figure~\ref{PSFFig3ad}a. Let $a=\rho b +r$, where $\rho\geq 0$ and $0\leq r<b$. Then in its hybrid diagram, we drag $\rho+1$ times the vertex $A^-$
together with the edges meeting with the vertex $A^-$ over the $s$-connection on the $B$-handle. This performance
corresponds to a change of cutting disks inducing an automorphism of $\pi_1(H)$ that takes $A \mapsto AB^{-(\rho+1)s}$.
The resulting hybrid diagram of $R$ is depicted in Figure~\ref{PSFFig3ad}b, where there is only one band of connections labelled by $s$ in the $B$-handle.
Also in the $A$-handle, $R$ intersects the cutting disk $D_A$ positively. Note
that the curve $\beta$ remains same under the change of cutting disks.

If $r=0$, then from the equation $a=\rho b +r$, $b=1$ and $a=\rho$, whence $R$ consists of only one edge connecting the vertices $A^+$ and $A^-$ in Figure~\ref{PSFFig3ad}b. This implies that $R$ and $\beta$ have the R-R diagram of Figure~\ref{DPCFig21b}a.

Now suppose $r>0$. The corresponding R-R diagram of $R$ has the form of either Figure~\ref{DPCFig21as} or Figure~\ref{DPCFig21b}b depending on
the number of bands of connections that $R$ has in the $A$-handle. In either case,
since $b-r>0$, there must be bands of connections whose label is greater than $1$, which indicates that there is no $0$-connection in the $A$-handle.
By Theorem~\ref{at least one of a, b, c is zero} and Lemma~\ref{may assume (m,n) ?(1,1)}
we may assume that the R-R diagram of $R$ and $\beta$ has the form of Figure~\ref{DPCFig21b}b. Since there is no $0$-connection in the $A$-handle, we apply
exactly same argument as that in fifth and seventh paragraphs in this proof to conclude that the R-R diagram of $R$ and $\beta$ is equivalent to either the R-R diagram of Figure~\ref{DPCFig21b}a or the R-R diagram of Figure~\ref{DPCFig21b}b with $m,n>0$. This completes the proof.
\end{proof}

By using Theorem~\ref{main theorem 1}, we can prove the following, which is Theorem~\ref{main theorem2}.

\begin{thm}
\label{one consequence}
Suppose $R$ is a nonseparating simple closed curve in the boundary of a genus two handlebody such that $H[R]$ embeds in $S^3$. Then there exists a proper power curve disjoint from $R$ if and only if $H[R]$ is the exterior of the unknot, a torus knot, or a tunnel-number-one cable of a torus knot.
\end{thm}

\begin{proof}
The ``if" part immediately follows from Theorem~\ref{main theorem 1}.

For the ``only if" part, assume that $H[R]$ is the exterior of the unknot, a torus knot, or a tunnel-number-one cable of a torus knot. For the last two cases, by Propositions~\ref{the converse true for cable knots} and \ref{the converse true for torus knots} $R$ has an R-R diagram of the form shown in Figure~\ref{DPCFig21b}b with $a,b\geq 0$ and $s>1$. Hence there exists a proper power curve disjoint from $R$.

If $H[R]$ is the exterior of the unknot, then $H[R]$ is a solid torus, which implies that $R$ is primitive in $H$. Therefore, there exists a complete set of cutting disks $\{D_A, D_B\}$ of $H$ such that $R$ has an R-R diagram of the form shown in Figure~\ref{DPCFig21b}a with $s>1$. So there exists a proper power curve disjoint from $R$, completing the proof.
\end{proof}


\end{document}